\documentclass[10pt,a4paper,oneside]{article}
\usepackage[english]{babel}
\usepackage{a4wide,amsmath,amsthm,amssymb,url,graphicx,xspace,algorithm,algorithmic,pgf}
\usepackage[latin1]{inputenc}
\usepackage{enumitem}
\usepackage{multirow}
\usepackage{hyperref}
\usepackage{tabularx}
\usepackage{tikz}
\usepackage{allrunes}
\usepackage{bm}

\def\F{\mathbb{F}}

\def\N{\mathbb{N}}
\def\R{\mathbb{R}}

\DeclareMathOperator{\PG}{PG}

\DeclareMathOperator{\col}{col}

\DeclareMathOperator{\im}{im}

\theoremstyle{definition}
\newtheorem{theorem}{Theorem}[section]
\newtheorem{lemma}[theorem]{Lemma}
\newtheorem{definition}[theorem]{Definition}
\newtheorem{remark}[theorem]{Remark}
\newtheorem{corollary}[theorem]{Corollary}
\newtheorem{notation}[theorem]{Notation}
\newtheorem{example}[theorem]{Example}

\newcommand{\comments}[1]{}
\newcommand{\gs}[3]{\genfrac{[}{]}{0pt}{0}{#1}{#2}_{#3}}

\author{
 Maarten De Boeck\footnote{Address: UGent, Department of Mathematics: Algebra and Geometry, Krijgslaan 281 -- S25, 9000 Gent, Flanders, Belgium. \newline Email address: Maarten.DeBoeck@UGent.be},
 Morgan Rodgers\footnote{Address: California State University, Fresno, Department of Mathematics, Fresno, CA, United States\newline Email address: morgan@csufresno.edu, Website: http://cage.ugent.be/$\sim$mrodgers/},
 Leo Storme\footnote{Address: UGent, Department of Mathematics: Analysis, Logic and Discrete Mathematics, Krijgslaan 281 -- S8, 9000 Gent, Flanders, Belgium. \newline Email address: Leo.Storme@UGent.be, Website: http://cage.ugent.be/$\sim$ls} and
 Andrea \v{S}vob\footnote{Address: University of Rijeka, Department of Mathematics, Radmile Matej\v{c}i\'c 2, 51000 Rijeka, Croatia.  \newline Email address: asvob@math.uniri.hr, Website: http://www.math.uniri.hr/$\sim$asvob/}}
\title{Cameron-Liebler sets of generators in finite classical polar spaces}
\date{}
\begin{document}
\maketitle
\begin{abstract}
	Cameron-Liebler sets were originally defined as collections of lines (``line classes'') in $\PG(3,q)$ sharing certain properties with line classes of symmetric tactical decompositions. While there are many equivalent characterisations, these objects are defined as sets of lines whose characteristic vector lies in the image of the transpose of the point-line incidence matrix of $\PG(3,q)$, and so combinatorially they behave like a union of pairwise disjoint point-pencils. Recently, the concept of a Cameron-Liebler set has been generalised to several other settings. In this article we introduce Cameron-Liebler sets of generators in finite classical polar spaces. For each of the polar spaces we give a list of characterisations that mirrors those for Cameron-Liebler line sets, and also prove some classification results.
\end{abstract}


\section{Introduction}\label{sec:introduction}
In~\cite{cl}, Cameron and Liebler investigated collineation groups of $\PG(n,q)$ having the same number of orbits on points as on lines. A line orbit of such a group possesses special properties, leading to the notion of what are now called Cameron-Liebler line classes in $\PG(n,q)$. These line sets have mostly been studied in $\PG(3,q)$ where many equivalent characterisations, both algebraic and combinatorial, are known; we refer to~\cite{penn} for an overview.
\par In recent years the concept of Cameron-Liebler classes has been generalised to many other contexts, for example Cameron-Liebler $k$-classes in $\PG(2k+1,q)$ were introduced in~\cite{rsv}, and Cameron-Liebler classes in finite sets were described in~\cite{dss}.
The central problem for Cameron-Liebler classes, regardless of the context, is to determine for which parameters $x$ a Cameron-Liebler class exists, and to classify the examples admitting a given parameter $x$.
Constructions of Cameron-Liebler line classes in $\PG(3,q)$ and classification results were obtained in~\cite{bd,cl,ddmr,gmp,TF:14,gmet,gmol,met,met2,rod}.
For Cameron-Liebler classes in finite sets a complete classification has been described in~\cite{dss}.
Classification results for Cameron-Liebler $k$-classes in $\PG(2k+1,q)$ were found in~\cite{met3,rsv}.
\par In this article we will introduce Cameron-Liebler sets of generators in polar spaces (we do not call them classes to avoid confusion with the classes of generators on a hyperbolic quadric). Our generalisation will be motivated by the following purely combinatorial definition for Cameron-Liebler line classes of $\PG(3,q)$.
\begin{definition}\label{CLLineClassDef}
 A set of lines $\mathcal{L}$ in $\PG(3,q)$ is a Cameron-Liebler line class with parameter $x$ if any line $\ell \in \mathcal{L}$ is disjoint to $(x-1)q^{2}$ lines of $\mathcal{L}$, and any line $\ell \notin\mathcal{L}$ is disjoint to $xq^{2}$ lines of $\mathcal{L}$.
\end{definition}
The collection of lines through a fixed point (a \emph{point-pencil}) contains $q^{2}+q+1$ lines, and is a Cameron-Liebler line class in $\PG(3,q)$ with parameter 1.
If we were able to take $\mathcal{L}$ to be a union of $x$ pairwise disjoint point-pencils, a line $\ell \in \mathcal{L}$ would be disjoint to $(x-1)q^{2}$ lines of $\mathcal{L}$, and a line $\ell \notin\mathcal{L}$ would be disjoint to $xq^{2}$ lines of $\mathcal{L}$, hence $\mathcal{L}$ would be a Cameron-Liebler line class with parameter $x$.
While it is impossible to have two disjoint point-pencils in $\PG(3,q)$, a Cameron-Liebler line class in $\PG(3,q)$ with parameter $x$ can be thought of as a set of lines that combinatorially (with respect to disjointness) behaves like a union of $x$ pairwise disjoint point-pencils.
\par Section~\ref{sec:prelim} will introduce the necessary background while Section~\ref{sec:characterisation} is devoted to proving the characterisation results, which will differ among the polar spaces. We will characterise the Cameron-Liebler sets using disjointness, as vectors in an eigenspace, as vectors in the image of a matrix and through their relation with spreads. In Section~\ref{sec:examples} we will present some examples of Cameron-Liebler sets of generators on polar spaces and in Section~\ref{sec:classification} we will present some classification results. The main result is Theorem \ref{mainclassification} in which the Cameron-Liebler sets with parameter at most $q^{e-1}+1$ are classified. Section~\ref{sec:quadrangles} focuses on the rank 2 case, the generators (lines) of generalised quadrangles.

\section{Preliminaries}\label{sec:prelim}

\subsection*{Distance-regular graphs and association schemes}
Our work will rely on many tools from the theory of association schemes, which were introduced in~\cite{bs}. We also refer to~\cite[Chapter 2]{bcn} and~\cite{del1} as important texts on this topic.
While we will follow the approach given in~\cite{vanhove1}, we will primarily focus on how these concepts apply to the specialized case of distance-regular graphs.
More information on this important class of association schemes can be found in~\cite{bcn}.

Throughout this paper, we will denote the matrix of all ones by $J$. The all-one vector of length $m$ will be denoted by $\bm{j_{m}}$; we denote it by $\bm{j}$ if its length is clear from the context.

\begin{definition}
 A \emph{$d$-class association scheme} $\mathcal{R}$ on a finite non-empty set $\Omega$ is a set $\{R_{0},\dots,R_{d}\}$ of $d+1$ symmetric relations on $\Omega$ such that the following properties hold.
 \begin{enumerate}
  \item $R_{0}$ is the identity relation.
  \item $\mathcal{R}$ is a partition of $\Omega^{2}$.
  \item For $x,z\in\Omega$ with $(x,z)\in R_{k}$, the number of $y\in\Omega$ such that $(x,y)\in R_{i}$ and $(y,z)\in R_{j}$, equals a constant $p^{k}_{ij}$ (and is thus independent of $x$ and $z$).
 \end{enumerate}
\end{definition}

Let $\Gamma$ be a graph with diameter $d$, and let $\delta(u,v)$ denote the distance between vertices $u$ and $v$ of $\Gamma$.
The $i$th-neighborhood of a vertex $v$ is the set $\Gamma_{i}(v) = \{ w : d(v,w) = i \}$; we similarly define $\Gamma_{i}$ to be the $i$th distance graph of $\Gamma$, that is, the vertex set of $\Gamma_{i}$ is the same as for $\Gamma$, with adjacency in $\Gamma_{i}$ defined by the $i$th distance relation in $\Gamma$.
We say that $\Gamma$ is \emph{distance-regular} if the distance relations of $\Gamma$ give the relations of a $d$-class association scheme, that is, for every choice of $0 \leq i,j,k \leq d$, all vertices $v$ and $w$ with $\delta(v,w)=k$ satisfy $|\Gamma_{i}(v) \cap \Gamma_{j}(w)| = p^{k}_{ij}$ for some constant $p^{k}_{ij}$.

In a distance-regular graph, we have that $p^{k}_{ij}=0$ whenever $i+j < k$ or $k<|i-j|$.
A distance-regular graph $\Gamma$ is necessarily regular with degree $p^{0}_{11}$; more generally, each distance graph $\Gamma_{i}$ is regular with degree $k_{i}=p^{0}_{ii}$.
An equivalent definition of a distance-regular graph is the existence of the constants $b_{i}=p^{i}_{i+1,1}$ and $c_{i}= p^{i}_{i-1,1}$ for $0 \leq i \leq d$ (notice that $b_{d}=c_{0}=0$).
This collection of values $\{b_{0}, \ldots, b_{d-1};c_{1}, \ldots, c_{d}\}$ is called the \emph{parameter set} of the distance-regular graph, and all of the remaining $p^{k}_{ij}$ are determined by these values.

\par For the remainder of this section, $\Gamma$ will be a distance-regular graph with diameter $d$ and vertex set $\Omega=\{x_{1},\dots,x_{n}\}$.
Let $A_{k}$ be the adjacency matrix of $\Gamma_{k}$ for $0 \leq k \leq d$, that is, $A_{k}$ will have a $1$ in the $(i,j)$ position if $\delta(x_{i}, x_{j})=k$ and $0$ otherwise.
This gives a set of $d+1$ symmetric matrices $A_{0},\dots,A_{d}$ with $A_{0}=I$, $\sum^{d}_{i=0}A_{i}=J$, and $A_{i}A_{j}=\sum^{d}_{k=0}p^{k}_{ij}A_{k}$.
The ordering on $\Omega$ used to form the matrices $A_{i}$ will be considered to be fixed, giving us an ordered basis for the vector space $\R^{\Omega}$.
\par The $\R$-vector space generated by the matrices $A_{0},\dots,A_{d}$ is closed under matrix multiplication and so it is a $(d+1)$-dimensional commutative $\R$-algebra $\mathcal{A}$, which is known as the \emph{Bose-Mesner algebra} of the association scheme. We can also obtain a second important basis for the Bose-Mesner algebra.
A matrix $E$ is \emph{idempotent} if $E^{2}=E$, and an idempotent matrix is called \emph{minimal} if it cannot be written as the sum of two nonzero idempotent matrices.

\begin{theorem}\label{idempotents}
 The Bose-Mesner algebra of a $d$-class association scheme on $\Omega$ has a unique basis $\{E_{0},\dots,E_{d}\}$ of minimal idempotents. Moreover, $E_{i}E_{j}=\delta_{ij}E_{i}$.
 The subspaces $V_{0},\dots,V_{d}$, with $V_{i}=\im(E_{i})$, form an orthogonal decomposition of $\R^{\Omega}$.
\end{theorem}

Since $|\Omega|=n$, we always consider $E_{0}=\frac{1}{n}J$. Since $\{A_{i}\}$ and $\{E_{i}\}$ are both bases for $\mathcal{A}$, there are values $P_{j,i}$ such that $A_{i}=\sum^{d}_{j=0}P_{j,i}E_{j}$.
From this it is clear that $A_{i}E_{j} = P_{j,i}E_{j}$, therefore $V_{j}=\im(E_{j})$ is a subspace of eigenvectors for the matrix $A_{i}$ with corresponding eigenvalue $P_{j,i}$ (the $P_{j,i}$ are thus called the \emph{eigenvalues} of the association scheme).
\par Note that $V_{j}$ is not necessarily an eigenspace for all incidence matrices as it could happen that for a particular matrix $A_{i}$, two eigenvalues coincide.
E.g., for the matrix $A_{0}$ there is only one eigenvalue and hence $\R^{\Omega}$ is the only eigenspace.
However since it is true that each eigenspace is an orthogonal sum of the $V_{j}$, and $\Gamma$ (with adjacency matrix $A_{1}$) has exactly $d+1$ distinct eigenvalues, we refer to the $V_{j}$ as the \emph{eigenspaces} of the underlying association scheme.

The eigenspaces of an association scheme give us a powerful tool for working with the Bose-Mesner algebra, as is shown by the following result which holds for any association scheme.
\begin{lemma}\label{rowspacesubspacesum}
 Let $\mathcal{R}$ be a $d$-class association scheme on the set $\Omega$ and let $\mathcal{A}$ be the corresponding Bose-Mesner algebra. If $A\in\mathcal{A}$, then $\im(A)$ is a direct sum of eigenspaces of the association scheme.
\end{lemma}
\begin{proof}
 All matrices in $\mathcal{A}$ are symmetric, so $\im(A) = \im(A^{t})$.
 Let $E_{0},\dots,E_{d}$ be the minimal idempotents of $\mathcal{A}$, giving a corresponding orthogonal decomposition $\R^{\Omega} = V_{0}\perp\dots\perp V_{d}$. Since $A\in\mathcal{A}$, we can write $A=\sum_{i=0}^{d}c_{i}E_{i}$.
 We set $\mathcal{I}$ to be the collection of indices for which $c_{i}\neq0$.
 We will now show that $\im(A)=\left(\perp_{i\in\mathcal{I}}V_{i}\right)$. If $\bm{v}\in\im(A)$ there is a vector $\bm{w}$ such that $\bm{v}=A\bm{w}$, and thus
 \[
  \bm{v}=A\bm{w}=\left(\sum_{i=0}^{d}c_{i}E_{i}\right)\bm{w}=\sum_{i=0}^{d}c_{i}(E_{i}\bm{w})\in\left(\perp_{i\in\mathcal{I}}V_{i}\right)\;,
 \]
 since $V_{i}=\im(E_{i})$.
 On the other hand, for any $\bm{v}\in\left(\perp_{i\in\mathcal{I}}V_{i}\right)$ we have $\bm{v}=\sum_{i\in\mathcal{I}}\bm{v_{i}}$ with each $\bm{v_{i}}\in V_{i}$,
 and thus $\bm{v}=\sum_{i\in\mathcal{I}}E_{i}\bm{w_{i}}$ for some collection of vectors $\bm{w_{i}}\in\mathbb{R}^{\Omega}$.
 Thus we have that
 \[
  A\left(\sum_{i\in\mathcal{I}}\frac{1}{c_{i}}E_{i}\bm{w}_{i}\right)=\left(\sum_{j=0}^{d}c_{j}E_{j}\right)\left(\sum_{i\in\mathcal{I}}\frac{1}{c_{i}}E_{i}\bm{w}_{i}\right)=\sum_{i\in\mathcal{I}}\sum^{d}_{j=0}\frac{c_{j}}{c_{i}}\delta_{ij}E_{i}\bm{w_{i}}=\sum_{i\in\mathcal{I}}E_{i}\bm{w_{i}}=\bm{v}\;,
 \]
 and so $\bm{v}\in\im{(A)}$.
\end{proof}

We also mention the following result about cliques for a relation in an association scheme.
\begin{theorem}[{\cite[Corollary 2.2.9]{vanhove1}}]\label{clique}
 Let $R$ be a relation in an association scheme and let $S$ be a clique of $R$ with characteristic vector $\bm{\chi}$. Let $V$ be an eigenspace of the scheme, with $\lambda$ the eigenvalue of $R$ on $V$, and $k$ the eigenvalue of $V$ corresponding to the identity relation. If $\lambda<0$, then $|S|\leq 1-\frac{k}{\lambda}$, and in case $|S|=1-\frac{k}{\lambda}$ then $\bm{\chi}\in V^{\perp}$.
\end{theorem}

\subsection*{Polar spaces}

Polar spaces of rank $d\geq 2$ are a kind of incidence geometries having subspaces of dimension $0,\dots,d-1$.
In this article we will only discuss finite polar spaces (polar spaces having a finite number of subspaces).
The axiomatic introduction of polar spaces is due to Veldkamp \cite{veld1,veld2} and Tits \cite{tits}.
Polar spaces of rank 2 are called generalised quadrangles. We define them separately.

\begin{definition}\label{defgq}
 A \emph{generalised quadrangle of order $(s,t)$}, with $s,t\geq1$, is a point-line geometry satisfying the following axioms.
 \begin{itemize}
  \item Two distinct points are incident with at most one line.
  \item Any line is incident with $s+1$ points and any point is incident with $t+1$ lines.
  \item For any point $P$ and line $\ell$, with $P$ not incident with $\ell$, there is a unique tuple $(P',\ell')$ such that $P$ is incident with $\ell'$, $P'$ is incident with $\ell$ and $P'$ is incident with $\ell'$.
 \end{itemize}
\end{definition}
For any generalised quadrangle $\mathcal{Q}$ of order $(s,t)$ we can define the \emph{dual generalised quadrangle} $\mathcal{Q}^{D}$ by interchanging the roles of points and lines in $\mathcal{Q}$. The generalised quadrangle $\mathcal{Q}^{D}$ has order $(t,s)$.
\par The reason for introducing the polar spaces of rank 2 separately is the famous result by Tits (see \cite{tits}) which states that all finite polar spaces of rank at least $3$ are \emph{classical polar spaces}. Therefore we immediately introduce finite classical polar spaces.

\begin{definition}
 A \emph{finite classical polar space} is an incidence geometry consisting of the totally isotropic subspaces of a non-degenerate quadratic or non-degenerate reflexive sesquilinear form on a vector space $\F^{n+1}_{q}$.
\end{definition}

Throughout this article we will consider the finite classical polar spaces as substructures of $\PG(n,q)$, and describe all dimensions as projective dimensions. The subspaces of dimension $0$ (vector lines), $1$ (vector planes) and $2$ (vector solids) are called \emph{points}, \emph{lines} and \emph{planes}, respectively.

\par A bilinear form for which all vectors are isotropic is called \emph{symplectic}. A sesquilinear form on $V$ is called \emph{Hermitian} if the corresponding field automorphism $\theta$ is an involution and $f(v,w)={f(w,v)}^{\theta}$ for all $v,w\in V$. We now list the finite classical polar spaces of rank $d$.
\begin{itemize}
 \item The hyperbolic quadric $\mathcal{Q}^{+}(2d-1,q)$ embedded in $\PG(2d-1,q)$. It arises from a hyperbolic quadratic form on $V(2d,q)$. Its standard equation is $X_{0}X_{1}+\dots+X_{2d-2}X_{2d-1}=0$.
 \item The parabolic quadric $\mathcal{Q}(2d,q)$ embedded in $\PG(2d,q)$. It arises from a parabolic quadratic form on $V(2d+1,q)$. Its standard equation is $X^{2}_{0}+X_{1}X_{2}+\dots+X_{2d-1}X_{2d}=0$.
 \item The elliptic quadric $\mathcal{Q}^{-}(2d+1,q)$ embedded in $\PG(2d+1,q)$. It arises from an elliptic quadratic form on $V(2d+2,q)$. Its standard equation is $X_{0}X_{1}+\dots+X_{2d-2}X_{2d-1}+g(X_{2d},X_{2d+1})=0$ with $g$ a homogeneous irreducible quadratic polynomial over $\F_{q}$.
 \item The Hermitian polar space $\mathcal{H}(2d-1,q^{2})$ embedded in $\PG(2d-1,q^{2})$. It arises from a Hermitian form on $V(2d,q^{2})$, constructed using the field automorphism $x\mapsto x^{q}$. Its standard equation is $X^{q+1}_{0}+X^{q+1}_{1}+\dots+X^{q+1}_{2d-1}=0$.
 \item The Hermitian polar space $\mathcal{H}(2d,q^{2})$ embedded in $\PG(2d,q^{2})$. It arises from a Hermitian form on $V(2d+1,q^{2})$, constructed using the field automorphism $x\mapsto x^{q}$. Its standard equation is $X^{q+1}_{0}+X^{q+1}_{1}+\dots+X^{q+1}_{2d}=0$.
 \item The symplectic polar space $\mathcal{W}(2d-1,q)$ embedded in $\PG(2d-1,q)$. It arises from a symplectic form on $V(2d,q)$. For this symplectic form we can choose an appropriate basis $\{e_{1},\dots,e_{d},e'_{1},\dots,e'_{d}\}$ of $V(2d,q)$ such that $f(e_{i},e_{j})=f(e'_{i},e'_{j})=0$ and $f(e_{i},e'_{j})=\delta_{i,j}$, with $1\leq i,j\leq d$.
\end{itemize}

\begin{definition}
 The subspaces of maximal dimension (being $d-1$) of a polar space of rank $d$ are called \emph{generators}. We define the \emph{parameter} $e$ of a polar space $\mathcal{P}$ over $\F_{q}$ as $\log_{q}(x-1)$ with $x$ the number of generators through a $(d-2)$-space of $\mathcal{P}$.
\end{definition}

The parameter of a polar space only depends on the type of the polar space and not on its rank. Table \ref{parameter} gives an overview.

\begin{table}[ht]
 \centering
 \begin{tabular}{|c|c|}
  \hline
  polar space               & $e$ \\ \hline\hline
  $\mathcal{Q}^{+}(2d-1,q)$ & 0   \\ \hline
  $\mathcal{H}(2d-1,q)$     & 1/2 \\ \hline
  $\mathcal{W}(2d-1,q)$     & 1   \\ \hline
  $\mathcal{Q}(2d,q)$       & 1   \\ \hline
  $\mathcal{H}(2d,q)$       & 3/2 \\ \hline
  $\mathcal{Q}^{-}(2d+1,q)$ & 2   \\ \hline
 \end{tabular}
 \caption{The parameter of the finite classical polar spaces.}
 \label{parameter}
\end{table}

\par Now we discuss a remarkable property of the hyperbolic quadrics.

\begin{remark}\label{hyperclass}
 We define a relation $\sim$ on the set $\Omega$ of generators of a hyperbolic quadric $\mathcal{Q}^{+}(2d+1,q)$ in the following way: $\pi\sim\pi'\Leftrightarrow\dim(\pi\cap\pi')\equiv d\pmod{2}$. This relation is an equivalence relation and it has two equivalence classes, which necessarily partition the set $\Omega$. These equivalence classes are commonly called the \emph{Greek} and \emph{Latin} generators.
 \par We note that if $d$ is even, then two generators in the same class cannot be disjoint. If $d$ is odd, then two generators from different classes cannot be disjoint.
\end{remark}

We now introduce the \emph{Gaussian binomial coefficient} $\gs{n}{k}{q}$ for positive integers $n,k$ and prime power $q\geq2$:
\[
 \gs{n}{k}{q}=\prod^{k}_{i=1}\frac{q^{n-k+i}-1}{q^{i}-1}=\frac{(q^{n}-1)\cdots(q^{n-k+1}-1)}{(q^{k}-1)\cdots(q-1)}\;.
\]
The number $\gs{n}{k}{q}$ equals the number of $k$-dimensional subspaces of the vector space $\F^{n}_{q}$, equivalently the number of $(k-1)$-spaces in $\PG(n-1,q)$.
The Gaussian binomial coefficients admit the following relation $\gs{n}{k}{q}=\gs{n}{n-k}{q}$, which follows immediately from duality. Another important result regarding Gaussian binomial coefficients is the following analogue of the binomial theorem.
\begin{lemma}\label{qbinomialtheorem}
 For integers $n,k\geq0$ and prime power $q\geq2$ we have
 \[
  \sum^{n}_{k=0}\gs{n}{k}{q}q^{\binom{k}{2}}t^{k}=\prod^{n-1}_{k=0}(1+q^{k}t)\;.
 \]
\end{lemma}

We can express the number of subspaces of a polar space given its parameter and rank using Gaussian binomial coefficients.

\begin{lemma}[{{\cite[Lemma 9.4.1]{bcn}}}]\label{subspacesonpolar}
 For a finite classical polar space $\mathcal{P}$ of rank $d$ with parameter $e$, embedded in a projective space over $\F_{q}$, the number of $k$-spaces of  $\mathcal{P}$ is given by
 \[
  \gs{d}{k+1}{q}\prod^{k+1}_{i=1}(q^{d+e-i}+1).
 \]
\end{lemma}

In particular, the number of generators of a finite classical polar space of rank $d$ with parameter $e$, embedded in a projective space over $\F_{q}$, equals $\prod^{d-1}_{i=0}(q^{e+i}+1)$, and the number of points equals $\gs{d}{1}{q}(q^{d+e-1}+1)$.

\begin{corollary}\label{pointpencilsize}
 For a finite classical polar space $\mathcal{P}$ of rank $d$ with parameter $e$, embedded in a projective space over $\F_{q}$, the number of $k$-spaces of $\mathcal{P}$ through a fixed $m$-space is given by
 \[
  \gs{d-m-1}{k-m}{q}\prod^{k-m}_{i=1}(q^{d-m+e-i-1}+1).
 \]
\end{corollary}

In particular, the number of generators through a fixed point of a finite classical polar space of rank $d$ with parameter $e$, embedded in a projective space over $\F_{q}$, equals $\prod^{d-2}_{i=0}(q^{e+i}+1)$.

\subsection*{Distance-regular graphs from generators of finite polar spaces}
We can use the set of generators of a finite classical polar space to construct a distance-regular graph.

\begin{definition}
 For a finite classical polar space $\mathcal{P}$ of rank $d$ we define the corresponding \emph{dual polar graph} as follows: its vertices are the generators of $\mathcal{P}$ and two vertices are adjacent if the corresponding generators meet in a $(d-2)$-space.
\end{definition}

\begin{theorem}\label{dualpolargraphs}
 The dual polar graph of a finite classical polar space over $\F_{q}$ of rank $d$ and with parameter $e$ is distance-regular with diameter $d$ and parameter set $\{b_{0},\dots,b_{d-1};c_{1},\dots,c_{d}\}$ with
 \[
  b_{i}=q^{i+e}\gs{d-i}{1}{q}\;,\quad 0\leq i\leq d-1\qquad\text{and}\qquad c_{i}=\gs{i}{1}{q}\;,\quad 1\leq i\leq d\;.
 \]
 Two vertices are at distance $j$ if the corresponding generators meet in a $(d-j-1)$-space.
\end{theorem}

\par Using the terminology of association schemes, we will take a finite classical polar space $\mathcal{P}$ of rank $d$ and parameter $e$ and let $\Omega$ be its set of generators (with some fixed ordering). Note that, by Lemma~\ref{subspacesonpolar}, $|\Omega| = \prod_{i=0}^{d-1}(q^{e+i}+1)$.
We will consider $\Gamma$ to be the dual polar graph of $\mathcal{P}$, giving $d+1$ distance graphs $\Gamma_{i}$ with adjacency matrices $A_{i}$ for $0 \leq i \leq d$.
This gives an orthogonal decomposition $\R^{\Omega} = V_{0}\perp V_{1}\perp\dots\perp V_{d}$ of $\R^{\Omega}$ where the $V_{i}$ are the eigenspaces of the association scheme defined by $\Gamma$.
Recall that $V_{0} = \left\langle\bm{j}\right\rangle$.
The remaining eigenspaces $V_{1}$, $\ldots$, $V_{d}$ are ordered using the following result described in~\cite{vanhove2}, originating in~\cite{del2}.
\begin{theorem}\label{rowAissumeigenspaces}
 Let $\mathcal{P}$ be a finite classical polar space of rank $d$ and, for $1 \leq t \leq d$, let $C_{k}$ be the incidence matrix between the $(k-1)$-spaces and the generators of $\mathcal{P}$. Then the eigenspaces $V_{1}$, $\ldots$, $V_{d}$ of the association scheme of the dual polar graph for $\mathcal{P}$ can be ordered so that
 $\im(C_{k}^{t}) = V_{0} \perp \cdots \perp V_{k}$. In particular, $V_{k} = \im(C_{k}^{t}) \cap \ker(C_{k-1})$.
\end{theorem}

Then the eigenvalue of the $i$th distance graph $\Gamma_{i}$ corresponding to the subspace $V_{j}$ is given by
\[
 P_{j,i}=\sum^{\min(j,d-i)}_{u=\max(0,j-i)}{(-1)}^{j+u}\gs{d-j}{d-i-u}{q}\gs{j}{u}{q}q^{\binom{u+i-j}{2}+\binom{j-u}{2}+e(u+i-j)}\;.
\]

The \emph{disjointness graph} $\Gamma_{d}$ will have special importance for our work; this is the graph for which two vertices are adjacent if and only if the corresponding generators of the polar space have a trivial intersection.
The eigenvalues of $\Gamma_{d}$ are given by
\[
 P_{j,d}={(-1)}^{j}q^{\binom{d}{2}+(d-j)(e-j)}\;.
\]
\begin{lemma}\label{skewgenerators}
 Let $\mathcal{P}$ be a finite classical polar space over $\F_{q}$ of rank $d$ and with parameter $e$. The number of generators disjoint to a given generator equals $q^{\binom{d}{2}+de}$.
\end{lemma}
\begin{proof}
 This number is given by the eigenvalue of the disjointness graph $\Gamma_{d}$ on the eigenspace $V_{0}$, so the result follows from the above formula for $P_{0,d}$.
\end{proof}
We will frequently be interested in the eigenspace of the minimal eigenvalue of $\Gamma_{d}$. Most of the following details can be found in the proof of~\cite[Theorem 8]{psv} and in~\cite[Corollary 4.3.17]{vanhove1}.

\begin{theorem}\label{mineigenvalue}
 Let $\mathcal{P}$ be a finite classical polar space over $\F_{q}$ of rank $d$ and with parameter $e$. Let $V_{0}\perp V_{1}\perp\dots\perp V_{d}$ be the orthogonal decomposition in eigenspaces of $\R^{\Omega}$ with $\Omega$ the set of generators of $\mathcal{P}$, as above.
 The eigenvalue of the disjointness relation of the generators for the subspace $V_{1}$ is $m=-q^{\binom{d-1}{2}+e(d-1)}$.
 \begin{itemize}
  \item If $\mathcal{P}=\mathcal{Q}^{+}(2d-1,q)$, with $d$ odd, then the minimal eigenvalue $-q^{\binom{d}{2}}$ appears only for $V_{d}$. The eigenvalue $m$ only appears for $V_{1}$.
  \item If $\mathcal{P}=\mathcal{H}(2d-1,q)$, $q$ a square, with $d$ odd, then the minimal eigenvalue $-q^{\binom{d}{2}}$ appears only for $V_{d}$. The eigenvalue $m$ only appears for $V_{1}$.
 \end{itemize}
 In all other cases $m$ is the minimal eigenvalue. In three cases, there are two subspaces $V_{i}$ for which it is an eigenvalue.
 \begin{itemize}
  \item If $\mathcal{P}=\mathcal{Q}^{+}(2d-1,q)$, with $d$ even, then the minimal eigenvalue $m$ appears for both $V_{1}$ and $V_{d-1}$.
  \item If $\mathcal{P}=\mathcal{Q}(2d,q)$, with $d$ odd, or $\mathcal{P}=\mathcal{W}(2d-1,q)$, with $d$ odd, then the minimal eigenvalue $m$ appears for both $V_{1}$ and $V_{d}$.
 \end{itemize}
 In all remaining cases, namely $\mathcal{P}=\mathcal{Q}(2d,q)$, with $d$ even, $\mathcal{P}=\mathcal{W}(2d-1,q)$, with $d$ even, $\mathcal{P}=\mathcal{Q}^{-}(2d+1,q)$, $\mathcal{P}=\mathcal{H}(2d-1,q)$, with $d$ even and $q$ a square, and $\mathcal{P}=\mathcal{H}(2d,q)$, the minimal eigenvalue is $m$ and it appears only for $V_{1}$.
\end{theorem}

\par We will also need the following result on the automorphism group of a finite classical polar space and the automorphism group of its dual polar graph.
\begin{theorem}[{\cite[Theorem 9.4.3 and Chapter 10, pp. 334--336]{bcn}}]\label{distancetransitive}
 The automorphism group of a finite classical polar space acts transitively on pairs of disjoint generators. The automorphism group of a dual polar graph is distance-transitive (acts transitively on the set of pairs of vertices at the same distance).
 \par These two groups coincide for all finite classical polar spaces except for the polar spaces $\mathcal{Q}^{+}(3,q)$.
\end{theorem}

For more information on dual polar graphs we refer to~\cite[Section 9.4]{bcn}.
The association schemes for $k$-dimensional subspaces of finite classical polar spaces were studied in~\cite{eis,sta,vanhove1} for arbitrary $k\leq d-1$.

\begin{remark}\label{assoschemehyperbolic}
 For the polar space $\mathcal{Q}^{+}(2d-1,q)$ the generators are naturally partitioned into two sets $\Omega_{1}$ and $\Omega_{2}$ (the Latin and Greek generators, see Remark~\ref{hyperclass}).
 For generators $\pi\in\Omega_{j}$ and $\pi'\in\Omega_{j'}$ with $(\pi,\pi')\in R_{i}$, $j,j'\in \{1,2\}$ and $i\in \{0,\dots,d\}$, we know that $j=j'$ if $i$ is even and that $j\neq j'$ if $i$ is odd.
 So, we can define a $\left\lfloor\frac{d}{2}\right\rfloor$-class association scheme for one class of generators $\Omega_{j}$ of $\mathcal{Q}^{+}(2d-1,q)$, namely $\mathcal{R}'=\left\{R'_{0},R'_{1},\dots,R'_{\left\lfloor\frac{d}{2}\right\rfloor}\right\}$ with $R'_{i}=R_{2i}$.
 The corresponding matrices are $A'_{0},A'_{1},\dots,A'_{\left\lfloor\frac{d}{2}\right\rfloor}$ with $A'_{i}$ equal to the matrix $A_{2i}$ restricted to the rows and columns of $\Omega_{j}$.
 Note that the eigenvalues $P_{j,i}$ and $P_{d-j,i}$ of the association scheme of $\mathcal{Q}^{+}(2d-1,q)$ are equal if $i$ is even, so the eigenspaces of $\mathcal{R}'$ are given by $V'_{j}\cap\R^{\Omega_{k}}$ with $V'_{j}=V_{j}\perp V_{d-j}$, $j=0,\dots,\left\lfloor\frac{d}{2}\right\rfloor$.
 Here $\R^{\Omega_{k}}$ is seen as the subspace of $\R^{\Omega}$ of all vectors which have non-zero entries only on the positions corresponding to the generators of $\Omega_{k}$.
\end{remark}

\subsection*{Spreads and $m$-regular systems of polar spaces}
We end this section by defining some substructures of polar spaces (including generalised quadrangles).

An \emph{$m$-regular system} of a polar space $\mathcal{P}$ is a set $\mathcal{S}$ of generators such that any point of $\mathcal{P}$ is contained in precisely $m$ elements of $\mathcal{S}$. A $1$-regular system is called a \emph{spread}.
Spreads of polar spaces have been studied intensively in the past decades, but the existence problem has not yet been solved for all polar spaces. An overview can be found in Table~\ref{spreadtable}. It includes the classical generalised quadrangles (generalised quadrangles that are polar spaces of rank $2$).
\par The number of generators in an $m$-regular system of a finite classical polar space over $\F_{q}$ of rank $d$ with parameter $e$ equals $m(q^{d+e-1}+1)$; this follows from basic counting. In particular, a spread is a set of $q^{d+e-1}+1$ pairwise disjoint generators.

\begin{table}[ht]
 \centering
 \begin{tabular}{|l|c||l|c|}
  \hline
  \textbf{Polar Space}                          & \textbf{Spread?} & \textbf{Polar Space}                                    & \textbf{Spread?} \\ \hline\hline
  $\mathcal{Q}^{-}(5,q)$                        & yes              & $\mathcal{Q}^{+}(4n+1,q)$, $n\geq1$                     & no               \\ \hline
  $\mathcal{Q}^{-}(2d+1,q)$, $d\geq3$, $q$ even & yes              & one class $\mathcal{Q}^{+}(4n+3,q)$, $q$ even           & yes              \\ \hline
  $\mathcal{Q}^{-}(2d+1,q)$, $d\geq3$, $q$ odd  &                  & one class $\mathcal{Q}^{+}(4n+3,q)$, $n\geq 2$, $q$ odd &                  \\ \hline
  $\mathcal{Q}(4n,q)$, $n\geq1$, $q$ even       & yes              & $\mathcal{W}(2d-1,q)$, $d\geq2$                         & yes              \\ \hline
  $\mathcal{Q}(4n,q)$, $n\geq1$, $q$ odd        & no               & $\mathcal{H}(2d-1,q^{2})$, $d\geq2$                     & no               \\ \hline
  $\mathcal{Q}(4n+2,q)$, $n\geq 2$, $q$ odd     &                  & $\mathcal{H}(2d,q^{2})$, $d\geq2$                       &                  \\ \hline
  $\mathcal{Q}(4n+2,q)$, $q$ even               & yes              & $\mathcal{H}(4,4)$                                      & no               \\ \hline%
 \end{tabular}
 \vspace*{0.3cm}

 \begin{tabular}{|l|l|c|}
  \hline
  \multicolumn{2}{|l|}{\textbf{Polar Space}} & \textbf{Spread?}                                    \\ \hline\hline
  \multirow{3}{*}{\shortstack{$\mathcal{Q}(6,q)$                                                   \\ one class $\mathcal{Q}^{+}(7,q)$}}                         & $q$ odd prime                                 & yes \\ \cline{2-3}
                                             & $q$ odd and $q\not\equiv1\pmod{3}$            & yes \\ \cline{2-3}
                                             & $q$ odd, $q$ non-prime and $q\equiv1\pmod{3}$ &     \\ \hline
 \end{tabular}
 \caption{Overview on the existence of spreads for the finite classical polar spaces. If the box in the column `Spread?' is open, no result on the existence is known.~\cite[Table 3]{dbkm}}\label{spreadtable}
\end{table}


\section{The characterisation theorems}\label{sec:characterisation}
As previously mentioned, the concept of a Cameron-Liebler set has been defined for finite sets and for finite projective geometries.
We will analogously introduce Cameron-Liebler sets for finite classical polar spaces, and provide characterisations of these objects. These characterisations however will vary depending on the particular polar space.

\par We will discuss vectors and matrices in this section and the following ones. Vectors will always be column vectors. For the sake of clarity we will denote vectors in boldface. For two vectors $\bm{v}$, $\bm{w}$, the inner product $\langle \bm{v}, \bm{w} \rangle$ is the standard dot product $\bm{v}^{t}\bm{w}$.
The characteristic vector of a subset $S'\subseteq S$ is a vector whose positions are indexed by the elements of $S$ and whose entries equal 1 if the element corresponding to the position of the entry is in $S'$, and 0 otherwise.
Throughout the introduction to this section, $\mathcal{P}$ will be a finite classical polar space over $\F_{q}$ of rank $d$ and with parameter $e$.
Continuing with the notation used for the association scheme of the dual polar graph, we will use $\Omega$ to denote the set of generators of $\mathcal{P}$, and $V_{0}$, $V_{1}$, $\ldots$, $V_{d}$ to denote the eigenspaces as defined in Theorem~\ref{rowAissumeigenspaces}.

Our main goal is to generalise Definition~\ref{CLLineClassDef} to sets of generators of $\mathcal{P}$. The idea is, if $\mathcal{L}$ is a collection of generators of $\mathcal{P}$, then we will call $\mathcal{L}$ a \emph{Cameron-Liebler set with parameter $x$} if any generator $\pi \in \mathcal{L}$ is disjoint to $(x-1)q^{\binom{d-1}{2}+e(d-1)}$ generators of $\mathcal{L}$, and any generator $\pi \notin \mathcal{L}$ is disjoint to $xq^{\binom{d-1}{2}+e(d-1)}$ generators of $\mathcal{L}$, i.e.~taking $\bm{\chi}$ to be the characteristic vector of $\mathcal{L}$, for any generator $\pi$ of $\mathcal{P}$ we have $\pi$ disjoint from $\left(x-(\bm{\chi})_{\pi}\right)q^{\binom{d-1}{2}+e(d-1)}$ elements of $\mathcal{L}$. This main property will lead to subtly different definitions in various polar spaces, but the following  result will be an important tool in investigating these objects.
\begin{lemma}\label{disjointequiveigenvector}
 Let $\mathcal{P}$ be a finite classical polar space over $\F_{q}$ of rank $d$ and with parameter $e$, and let $K$ be the generator disjointness matrix of $\mathcal{P}$. Let $\mathcal{L}$ be a set of generators of $\mathcal{P}$ with characteristic vector $\bm{\chi}$.
 The number of generators of $\mathcal{L}$ disjoint from each fixed generator $\pi$ of $\mathcal{P}$ equals $(x-{(\bm{\chi})}_{\pi})q^{\binom{d-1}{2}+e(d-1)}$ if and only if $\bm{\chi}-\frac{x}{q^{d+e-1}+1}\bm{j}$ is a vector in the eigenspace of $K$ for the eigenvalue $-q^{\binom{d-1}{2}+e(d-1)}$.
\end{lemma}
\begin{proof}
 Notice that the generator-disjointness matrix of $\mathcal{P}$ is the adjacency matrix $A_{d}$ of the association scheme for the dual polar graph of $\mathcal{P}$, and that $-q^{\binom{d-1}{2}+e(d-1)}=P_{1,d}$, the eigenvalue of $A_{d}$ corresponding to the eigenspace $V_{1}$.
 Let $V$ be the eigenspace of $K$ associated with this eigenvalue (which contains but is not necessarily equal to $V_{1}$).
 The entry on the position corresponding to a given generator $\pi$ of the vector $K\bm{\chi}$ is the number of generators disjoint to $\pi$ that are contained in $\mathcal{L}$.
 Hence, this number is $(x-{(\bm{\chi})}_{\pi})q^{\binom{d-1}{2}+e(d-1)}$ if and only if
 \[
  K\bm{\chi}=q^{\binom{d-1}{2}+e(d-1)}(x\bm{j}-\bm{\chi})\;.
 \]
 The number of ones in a row of $K$ equals the number of generators disjoint to a given generator, so by Lemma~\ref{skewgenerators} we know that
 \[
  K\bm{j}=q^{\binom{d}{2}+de}\bm{j}\;.
 \]
 Now, the vector $\bm{\chi}-\frac{x}{q^{d+e-1}+1}\bm{j}$ is contained in $V$ if and only if
 \begin{align*}
  K\left(\bm{\chi}-\frac{x}{q^{d+e-1}+1}\bm{j}\right) & =-q^{\binom{d-1}{2}+e(d-1)}\left(\bm{\chi}-\frac{x}{q^{d+e-1}+1}\bm{j}\right)                  \\
                                                      & =-q^{\binom{d-1}{2}+e(d-1)}\left(\bm{\chi}-x\bm{j}+\frac{xq^{d+e-1}}{q^{d+e-1}+1}\bm{j}\right) \\
                                                      & =q^{\binom{d-1}{2}+e(d-1)}(x\bm{j}-\bm{\chi})-\frac{x}{q^{d+e-1}+1}K\bm{j}\;.
 \end{align*}
 So, the vector $\bm{\chi}-\frac{x}{q^{d+e-1}+1}\bm{j}$ lies in $V$ if and only if
 \[
  K\bm{\chi}=q^{\binom{d-1}{2}+e(d-1)}(x\bm{j}-\bm{\chi})\;,
 \]
 which proves the result.
\end{proof}

\begin{definition}
 The set of all generators of a polar space through a fixed point is called a \emph{point-pencil}.
\end{definition}
From Corollary~\ref{pointpencilsize}, the size of a point-pencil in $\mathcal{P}$ is given by $\prod^{d-1}_{i=1}(q^{d+e-1}+1)$. If we let $\mathcal{L}$ be a point-pencil, then the characteristic vector $\bm{\chi}$ of $\mathcal{L}$ is a column of $A^{t}$, where $A$ is the point-generator adjacency matrix of $\mathcal{P}$.

\par Before presenting the results, we first discuss some lemmata that build up to the main characterisation theorems.

\begin{lemma}\label{spreadintersectiontonumberanddisjoint}
 Let $\mathcal{P}$ be a finite classical polar space over $\F_{q}$ of rank $d$ and with parameter $e$, that admits spreads, and let $\pi$ be a generator of $\mathcal{P}$.
 Let $G$ be an automorphism group of $\mathcal{P}$ that acts transitively on the pairs of disjoint generators of $\mathcal{P}$ and let $\mathcal{C}$ be a set of generator spreads of $\mathcal{P}$ which is a union of some orbits of spreads under the action of $G$.
 Let $\mathcal{L}$ be a set of generators of $\mathcal{P}$ with characteristic vector $\bm{\chi}$. If the intersection number $|\mathcal{L}\cap\mathcal{S}|$ equals $x$ for every spread $\mathcal{S}\in\mathcal{C}$, then
 \begin{itemize}
  \item[(i)] $|\mathcal{L}|=x\prod^{d-2}_{i=0}(q^{e+i}+1)$;
  \item[(ii)] the number of generators of $\mathcal{L}$ disjoint to $\pi$ equals $(x-{(\bm{\chi})}_{\pi})q^{\binom{d-1}{2}+e(d-1)}$.
 \end{itemize}
\end{lemma}
\begin{proof}
 Let $n_{i}$ be the number of spreads in $\mathcal{C}$ containing $i$ given disjoint generators, $i=1,2$, and let $n_{0}$ be the total number of spreads in $\mathcal{C}$. Now $G$ acts transitively on the pairs of disjoint generators. Consequently, the values $n_{i}$ are independent of the chosen generators (and thus well defined) for $i=0,1,2$, since $\mathcal{C}$ is a union of spread orbits.
 \par By counting the tuples $(\tau',S)$, with $S\in\mathcal{C}$ and $\tau'$ an element (generator) of $S$, we find that
 \[
  n_{0}(q^{e+d-1}+1)=n_{1}\prod^{d-1}_{i=0}(q^{e+i}+1) \quad\Leftrightarrow\quad n_{0}=n_{1}\prod^{d-2}_{i=0}(q^{e+i}+1)\;.
 \]
 Now, let $\tau$ be a generator of $\mathcal{P}$. By counting the tuples $(\tau',S)$, with $S\in\mathcal{C}$ containing $\tau$, and $\tau'$ an element of $S$ different (and thus necessarily disjoint) from $\tau$, and using Lemma~\ref{skewgenerators}, we find that
 \[
  n_{1}q^{e+d-1}=n_{2}q^{\binom{d}{2}+de} \quad\Leftrightarrow\quad n_{1}=q^{\binom{d-1}{2}+e(d-1)}n_{2}\;.
 \]
 \par In order to prove (i) we count the tuples $(\pi',S)$, with $S$ a spread of $\mathcal{P}$ in $\mathcal{C}$ and $\pi'$ a generator of $\mathcal{L}\cap S$. On the one hand, we find $n_{0}$ spreads in $\mathcal{C}$, each containing $x$ generators that belong to $\mathcal{L}$. On the other hand, each generator of $\mathcal{L}$ is contained in $n_{1}$ generators. Hence,
 \[
  |\mathcal{L}|=x\frac{n_{0}}{n_{1}}=x\prod^{d-2}_{i=0}(q^{e+i}+1)\;.
 \]
 \par Now we count the tuples $(\pi',S)$, with $S\in\mathcal{C}$ containing $\pi$, and $\pi'$ a generator of $\mathcal{L}\cap S$ disjoint to $\pi$. On the one hand, there are $n_{1}$ spreads $S$ in $\mathcal{C}$ containing $\pi$ and by the assumption, $\mathcal{L}\cap S$ contains $x$ generators, which are necessarily disjoint.
 So, there are $x-{(\bm{\chi})}_{\pi}$ possible choices for $\pi'$ given $S$, namely $x-1$ if $\pi\in\mathcal{L}$ and $x$ if $\pi\notin\mathcal{L}$. On the other hand, for every generator $\pi'$ we choose, there are $n_{2}$ spreads in $\mathcal{C}$ containing $\pi$ and $\pi'$. Consequently, the number of choices for $\pi'$ equals
 \[
  (x-{(\bm{\chi})}_{\pi})\frac{n_{1}}{n_{2}}=(x-{(\bm{\chi})}_{\pi})q^{\binom{d-1}{2}+e(d-1)}\;.
 \]
 This concludes the proof of part (ii).
\end{proof}

Note that the union of orbits in the lemma above might contain just one orbit.

\begin{remark}\label{spreadstransitive}
 For any finite classical polar space its full automorphism group acts transitively on the pairs of disjoint generators by Theorem~\ref{distancetransitive}, so there are groups that fulfil the requirement for the group $G$ in the above lemma.
\end{remark}
\par The following results will be useful in the proofs of the characterisation theorems.


\begin{lemma}\label{spreadineigenspaces}
 Let $\mathcal{P}$ be a finite classical polar space over $\F_{q}$ of rank $d$ and with parameter $e$ and let $V_{0}\perp V_{1}\perp\dots\perp V_{d}$ be the corresponding decomposition of $\R^{\Omega}$, with $\Omega$ the set of generators of $\mathcal{P}$, using the classical ordering. If $\mathcal{S}$ is a generator spread of $\mathcal{P}$ with characteristic vector $\bm{\chi}$, then $\bm{\chi}-\frac{1}{\prod_{i=0}^{d-2}(q^{e+i}+1)}\bm{j}\in (V_{0}\oplus V_{1})^{\perp}$. Moreover, if $e=0$ and $d$ is even, then $\bm{\chi}-\frac{1}{\prod_{i=0}^{d-2}(q^{e+i}+1)}\bm{j}\in (V_{0}\oplus V_{1}\oplus V_{d-1})^{\perp}$; if $e=1$ and $d$ is odd, then $\bm{\chi}-\frac{1}{\prod_{i=0}^{d-2}(q^{e+i}+1)}\bm{j}\in (V_{0}\oplus V_{1}\oplus V_{d})^{\perp}$.
\end{lemma}
\begin{proof}
 Note that $\mathcal{S}$ is a clique for the disjointness relation of the association scheme of $\mathcal{P}$. So, by Theorem \ref{clique}, we know that $|\mathcal{S}|\leq 1-\frac{P_{0,d}}{P_{j,d}}$ for the integers $j\in\{0,\dots,d\}$ such that $P_{j,d}<0$. Hence for $j$ odd, we have
 \[
  |\mathcal{S}|\leq1-\frac{q^{\binom{d}{2}+de}}{(-1)^{j}q^{\binom{d}{2}+(d-j)(e-j)}}=1+q^{j(d+e-j)}\;.
 \]
 Since $\mathcal{S}$ is a spread we have equality in this bound for $j=1$ and for $j=d+e-1$ (in case $e\leq1$ and $d+e-1$ is an odd integer). It now follows from the second part of Theorem \ref{clique} that  $\bm{\chi}\in V^{\perp}_{1}$, and moreover that also $\bm{\chi}\in V^{\perp}_{d-1}$ if $e=0$ and $d$ is even, and that also $\bm{\chi}\in V^{\perp}_{d}$ if $e=1$ and $d$ is odd. The result now follows from the observation that $\bm{\chi}=\frac{q^{d+e-1}+1}{\prod_{i=0}^{d-1}(q^{e+i}+1)}\bm{j}+\bm{v}$ with $\bm{v}$ a vector in $V_{0}^{\perp}$.
\end{proof}

A part of this result (namely $\bm{\chi}\in V^{\perp}_{1}$) was already proved in \cite[Theorem 4.4.14(ii)]{vanhove1}.

\begin{lemma}\label{spreadvector}
 Let $\mathcal{P}$ be a finite classical polar space over $\F_{q}$ of rank $d$ and with parameter $e$, and let $A$ be the point-generator incidence matrix of $\mathcal{P}$. Let $\mathcal{S}$ be a set of generators of $\mathcal{P}$ with characteristic vector $\bm{\chi}$. For any $m\in\N$, the generator set $\mathcal{S}$ is an $m$-regular system if and only if $\bm{\chi}-\frac{m}{\prod^{d-2}_{i=0}(q^{e+i}+1)}\bm{j}\in\ker(A)$.
\end{lemma}
\begin{proof}
 We denote the number of points of $\mathcal{P}$ by $n'$ and the number of generators of $\mathcal{P}$ by $n$. It is immediate that $A\bm{j_{n}}=\left(\prod^{d-2}_{i=0}(q^{e+i}+1)\right)\bm{j_{n'}}$ since every row in $A$ contains $\prod^{d-2}_{i=0}(q^{e+i}+1)$ ones as this is the number of generators through a point.
 \par Now, the set $\mathcal{S}$ is an $m$-regular system if and only if every point of $\mathcal{P}$ is contained in precisely $m$ generators of $\mathcal{S}$, hence if and only if $A\bm{\chi}=m\bm{j_{n'}}$. The first statement now follows from
 \[
  \bm{\chi}-\frac{m}{\prod^{d-2}_{i=0}(q^{e+i}+1)}\bm{j_{n}}\in\ker(A)\quad\Leftrightarrow\quad A\bm{\chi}=\frac{m}{\prod^{d-2}_{i=0}(q^{e+i}+1)}A\bm{j_{n}}\;.\qedhere
 \]
\end{proof}

We now present a first characterisation theorem.

\begin{theorem}\label{characterisationth}
 Let $\mathcal{P}$ be a polar space of rank $d$ and with parameter $e$ over $\F_{q}$. Let $\mathcal{L}$ be a set of generators of $\mathcal{P}$ with characteristic vector $\bm{\chi}$, and let $K$ be the generator disjointness matrix of $\mathcal{P}$. Let $V_{0}\perp\dots\perp V_{d}$ be the classical decomposition in eigenspaces of the association scheme on the generators of $\mathcal{P}$.
 Denote $\frac{|\mathcal{L}|}{\prod^{d-2}_{i=0}(q^{e+i}+1)}$ by $x$. The following statements are equivalent.
 \begin{itemize}
  \item[(i)] For each fixed generator $\pi$ of $\mathcal{P}$, the number of elements of $\mathcal{L}$ disjoint from $\pi$ equals $(x-{(\bm{\chi})}_{\pi})q^{\binom{d-1}{2}+e(d-1)}$.
  \item[(ii)] The vector $\bm{\chi}-\frac{x}{q^{d+e-1}+1}\bm{j}$ is contained in the eigenspace of $K$ for the eigenvalue $-q^{\binom{d-1}{2}+e(d-1)}$.
  \item[(iii)] The characteristic vector $\bm{\chi}$ is contained in $V_{0}\perp V_{1}\perp V_{d-1}$ if $d$ is even and $e=0$, it is contained in $V_{0}\perp V_{1}\perp V_{d}$ if $d$ is odd and $e=1$, and it is contained in $V_{0}\perp V_{1}$ in all other cases.
 \end{itemize}
 If $\mathcal{P}$ admits a spread, then also the next two statements are equivalent to the previous ones.
 \begin{itemize}
  \item[(iv)] $|\mathcal{L}\cap\mathcal{S}|=x$ for every spread $\mathcal{S}$ of $\mathcal{P}$.
  \item[(v)] $|\mathcal{L}\cap\mathcal{S}|=x$ for every spread $\mathcal{S}\in\mathcal{C}$ of $\mathcal{P}$, with $\mathcal{C}$ a class of spreads which is a union of some orbits under the action of an automorphism group of $\mathcal{P}$ that acts transitively on the pairs of disjoint generators of $\mathcal{P}$.
 \end{itemize}
\end{theorem}
\begin{proof}
 By Lemma~\ref{disjointequiveigenvector}, statements (i) and (ii) are equivalent. Now we assume that (ii) is valid. We consider the association scheme of the generators of $\mathcal{P}$, with matrices $I=A_{0},A_{1},\dots,A_{d}=K$ and corresponding decomposition $\R^{\Omega}=V_{0}\perp V_{1}\perp\dots\perp V_{d}$, with $\Omega$ the set of generators of $\mathcal{P}$, using the ordering as in Section~\ref{sec:prelim}. By Theorem~\ref{mineigenvalue}, $V_{1}$ is the unique subspace in this decomposition admitting $-q^{\binom{d-1}{2}+e(d-1)}$ as eigenvalue unless $d$ is even and $e=0$, or $d$ is odd and $e=1$. If $d$ is even and $e=0$, then $V_{1}\perp V_{d-1}$ admits $-q^{\binom{d-1}{2}+e(d-1)}$ as eigenvalue; if $d$ is odd and $e=1$, then $V_{1}\perp V_{d}$ admits $-q^{\binom{d-1}{2}+e(d-1)}$ as eigenvalue. It now follows from (ii) that
 \[
  \bm{\chi}-\frac{x}{q^{d+e-1}+1}\bm{j}\in \begin{cases}
   V_{1}\perp V_{d-1} & \text{if }d\text{ even and }e=0 \\
   V_{1}\perp V_{d}   & \text{if }d\text{ odd and }e=1  \\
   V_{1}              & \text{else}
  \end{cases}.
 \]
 Statement (iii) follows immediately as $V_{0}=\left\langle\bm{j}\right\rangle$.
 \par We now assume that statement (iii) is valid. Since $\left\langle\bm{\chi},\bm{j}\right\rangle=|\mathcal{L}|$, we know that $\bm{\chi}-\frac{x}{q^{d+e-1}+1}\bm{j}\in V_{1}\perp V'$ with $V'=V_{d-1}$ if $d$ is even and $e=0$, $V'=V_{d}$ if $d$ is odd and $e=1$, and $V'$ the null space otherwise. We find that
 \begin{align*}
  K\left(\bm{\chi}-\frac{x}{q^{d+e-1}+1}\bm{j}\right)=P_{1,d}\left(\bm{\chi}-\frac{x}{q^{d+e-1}+1}\bm{j}\right)=-q^{\binom{d-1}{2}+e(d-1)}\left(\bm{\chi}-\frac{x}{q^{d+e-1}+1}\bm{j}\right)\;,
 \end{align*}
 hence statement (ii).
 \par From now on we assume that $\mathcal{P}$ admits a spread. We note that by Remark~\ref{spreadstransitive} the existence of a spread of $\mathcal{P}$ implies the existence of a class $\mathcal{C}$ as in statement (v). This shows that statement (v) is not an empty statement, which is important for the following argument. We show that (iii) implies (iv) and that (v) implies (i). Since (iv) clearly implies (v), this shows that statements (iv) and (v) are both equivalent to the statements (i)--(iii). Now, we assume that (iii) is valid. Let $\mathcal{S}$ be a spread of $\mathcal{P}$, with characteristic vector $\bm{\chi}_{\mathcal{S}}$. Then,
 \begin{align*}
  |\mathcal{L}\cap\mathcal{S}| & =\left\langle\bm{\chi},\bm{\chi}_{\mathcal{S}}\right\rangle=\left\langle\bm{\chi}-\frac{x}{q^{d+e-1}+1}\bm{j},\bm{\chi}_{\mathcal{S}}\right\rangle+\frac{x}{q^{d+e-1}+1}\left\langle\bm{\chi},\bm{j}\right\rangle=\frac{x|\mathcal{S}|}{q^{d+e-1}+1}=x\;, 
 \end{align*}
 by Lemma \ref{spreadineigenspaces}. This proves (iv). By Lemma~\ref{spreadintersectiontonumberanddisjoint}(ii), statement (v) implies (i).
\end{proof}

\begin{definition}\label{clset}
 Let $\mathcal{P}$ be a finite classical polar space. A generator set $\mathcal{L}$ that fulfils one of the statements in Theorem~\ref{characterisationth} (and consequently all of them) is called a \emph{Cameron-Liebler set with parameter} $x=\frac{|\mathcal{L}|}{\prod^{d-2}_{i=0}(q^{e+i}+1)}$.
\end{definition}

We also want to give a description of these Cameron-Liebler sets using the image of an incidence matrix. For this we make a distinction between several types of polar spaces.

\subsection*{The polar spaces $\mathcal{Q}^{-}(2d+1,q)$, $\mathcal{Q}(4n,q)$, $\mathcal{Q}^{+}(4n+1,q)$, $\mathcal{W}(4n-1,q)$ and $\mathcal{H}(n,q^{2})$}

In this section we present the characterisation theorem using the image of an incidence matrix for almost all classes of finite classical polar spaces, namely the elliptic quadrics, the parabolic quadrics of even rank, the hyperbolic quadrics of odd rank, the symplectic polar spaces of even rank and the Hermitian polar spaces. To simplify the statement of the theorem we will use the following notation.

\begin{notation}
 The polar spaces $\mathcal{Q}^{-}(2d+1,q)$, $\mathcal{Q}(2d,q)$ with $d$ even,
 $\mathcal{Q}^{+}(2d-1,q)$ with $d$ odd, $\mathcal{W}(2d-1,q)$ with $d$ even,
 $\mathcal{H}(2d-1,q)$ with $q$ a square, and $\mathcal{H}(2d,q)$ with $q$ a square,
 are called the polar spaces of \emph{type I}.
\end{notation}

We can immediately state the result.

\begin{theorem}\label{characterisationthI}
 Let $\mathcal{P}$ be a polar space of type I, and let $A$ be the point-generator incidence matrix of $\mathcal{P}$. Let $\mathcal{L}$ be a set of generators of $\mathcal{P}$ with characteristic vector $\bm{\chi}$ and denote $\frac{|\mathcal{L}|}{\prod^{d-2}_{i=0}(q^{e+i}+1)}$ by $x$. The following three statements are equivalent.
 \begin{itemize}
  \item[(i)] $\bm{\chi}\in\im(A^{t})$.
  \item[(ii)] $\bm{\chi}\in{(\ker(A))}^{\perp}$.
  \item[(iii)] $\mathcal{L}$ is a Cameron-Liebler set.
 \end{itemize}
\end{theorem}
\begin{proof}
 Statement (i) and statement (ii) are equivalent since $\im(A^{t})={(\ker(A))}^{\perp}$. The equivalence of statements (i) and (iii) follows from Theorem \ref{rowAissumeigenspaces} and Theorem \ref{characterisationth}(iii).
\end{proof}

We have proved that the size of the intersection of a Cameron-Liebler set and a spread only depends on the parameter of the Cameron-Liebler set. If a spread exists, this is a defining property. We show that for regular systems a similar property holds.

\begin{corollary}\label{CLregularsystemI}
 Let $\mathcal{P}$ be a finite classical polar space of type I over $\F_{q}$ of rank $d$ and with parameter $e$, that admits an $m$-regular system $\mathcal{S}$. If $\mathcal{L}$ is a Cameron-Liebler set with parameter $x$, then $|\mathcal{L}\cap\mathcal{S}|=mx$.
\end{corollary}
\begin{proof}
 Denote the characteristic vectors of $\mathcal{S}$ and $\mathcal{L}$ by $\bm{\chi}_{\mathcal{S}}$ and $\bm{\chi}$, respectively. By Lemma~\ref{spreadvector} we can find a vector $\bm{w}\in\ker(A)$ such that $\bm{\chi}_{\mathcal{S}}=\bm{w}+\frac{m}{\prod^{d-2}_{i=0}(q^{e+i}+1)}\bm{j}$.
 Then
 \[
  |\mathcal{L}\cap\mathcal{S}|=\left\langle\bm{\chi},\bm{\chi}_{\mathcal{S}}\right\rangle=\left\langle \bm{\chi},\bm{w}\right\rangle+\frac{m}{\prod^{d-2}_{i=0}(q^{e+i}+1)}\left\langle\bm{\chi},\bm{j}\right\rangle=\frac{m\:|\mathcal{L}|}{\prod^{d-2}_{i=0}(q^{e+i}+1)}=mx\;,
 \]
 since $\left\langle \bm{\chi},\bm{w}\right\rangle=0$ by Theorem~\ref{characterisationthI}(ii).
\end{proof}
This property (as compared to property (v) in Theorem \ref{characterisationth}) is clearly not a defining property for $m=0$ or $m=\prod_{i=0}^{d-2}(q^{e+i}+1)$. While these examples could be considered degenerate since they are themselves Cameron-Liebler sets as well as $m$-regular systems, there are other cases where the converse fails to hold. The issue is that a set $\mathcal{S}$ is an $m$-regular system for some $m$ when its characteristic vector $\bm{\chi}$ lies in $V_{1}^{\perp}$; but it is possible that $\bm{\chi}$ is also in $V_{j}^{\perp}$ for some $j > 1$. Then, since the eigenspaces are invariant under the automorphism group $G$ of the polar space, any set whose characteristic vector lies in $V_{0} \perp V_{1} \perp V_{j}$ will have constant intersection size with $\mathcal{S}^{g}$ for every $g \in G$.
\par For example, a set of generators in $\mathcal{Q}^{+}(5,q)$ is a Cameron-Liebler set if and only if its characteristic vector is contained in $\im(A^{t}) = V_{0}\perp V_{1}$. If we let $\Pi_{1}$ and $\Pi_{2}$ be the sets of Latin and Greek planes, respectively, the characteristic vectors $\bm{\chi}_{\Pi_{1}}$ and $\bm{\chi}_{\Pi_{2}}$ of these sets lie in $V_{0}\perp V_{3}$. The sets $\Pi_{1}$ and $\Pi_{2}$ are $(q+1)$-regular systems. Furthermore, $\mathcal{C} = \{ \Pi_{1}, \Pi_{2} \}$ is invariant under the action of $G$. But any set $\mathcal{S}$ of generators whose characteristic vector is contained in $V_{0}\perp V_{1}\perp V_{2}$ will have constant intersection size with every element of $\mathcal{C}$; such a set is not necessarily a Cameron-Liebler set. In fact, in $\mathcal{Q}^{+}(5,2)$, a computational search finds examples of $2$-regular systems whose characteristic vectors lie in $V_{0}\perp V_{2}$, providing counterexamples to the possibility of generalising Theorem~\ref{characterisationth}~(v) to $m$-systems for $m=q+1=3$.

\subsection*{The polar spaces $\mathcal{Q}^{+}(4n-1,q)$}

In this section we investigate the Cameron-Liebler sets for the hyperbolic quadrics of even rank. We call them the polar spaces of \emph{type II}. The discussion of the Cameron-Liebler sets of generators in these polar spaces heavily relies on the fact that there are two classes of generators on a hyperbolic quadric. For the relation with spreads the following lemma is important.

\begin{lemma}\label{spreadsclasses}
 Let $\mathcal{S}$ be a generator spread of $\mathcal{Q}^{+}(4n-1,q)$, and let $\Omega_{1}$ and $\Omega_{2}$ be the two classes of generators. Then either $\mathcal{S}\subseteq\Omega_{1}$ or $\mathcal{S}\subseteq\Omega_{2}$.
\end{lemma}
\begin{proof}
 This is immediate, as two generators of a different class cannot have an empty intersection since the rank of $\mathcal{Q}^{+}(4n-1,q)$ is even.
\end{proof}

Because of the previous lemma we present a characterisation theorem similar to Theorems \ref{characterisationth} and \ref{characterisationthI} for one class of generators instead of the complete set of generators. Some lemmata that we have proved before have an analogue for this situation.

\begin{lemma}\label{disjointequiveigenvectorII}
 Let $\mathcal{G}$ be a class of generators of the hyperbolic quadric $\mathcal{Q}^{+}(2d-1,q)$, $d$ even, and let $K$ be the disjointness matrix of $\mathcal{G}$. Let $\mathcal{L}\subseteq\mathcal{G}$ be a set of generators of $\mathcal{P}$ with characteristic vector $\bm{\chi}$.
 The number of generators of $\mathcal{G}$ disjoint to a given generator $\pi$ of $\mathcal{G}$ equals $(x-{(\bm{\chi})}_{\pi})q^{\binom{d-1}{2}}$ if and only if $\bm{\chi}-\frac{x}{q^{d-1}+1}\bm{j}$ is a vector in the eigenspace of $K$ for the eigenvalue $-q^{\binom{d-1}{2}}$.
\end{lemma}
\begin{proof}
 Analogous to the proof of Lemma~\ref{disjointequiveigenvector}. Note that all generators of $\mathcal{Q}^{+}(2d-1,q)$, $d$ even, that are disjoint to a given generator $\pi\in\mathcal{G}$, belong to $\mathcal{G}$.
\end{proof}

\begin{lemma}\label{spreadvectorclass}
 Let $\mathcal{G}$ be a class of generators of the hyperbolic quadric $\mathcal{Q}^{+}(2d-1,q)$, $d$ even, and let $A'$ be the incidence matrix of points of $\mathcal{Q}^{+}(2d-1,q)$ and generators in $\mathcal{G}$. Let $\mathcal{S}\subseteq\mathcal{G}$ be a set of generators with characteristic vector $\bm{\chi}$. Then, the generator set $\mathcal{S}$ is an $m$-regular system of $\mathcal{Q}^{+}(2d-1,q)$ if and only if $\bm{\chi}-\frac{m}{\prod^{d-2}_{i=1}(q^{i}+1)}\bm{j}\in\ker(A)$.
\end{lemma}
\begin{proof}
 Analogous to the proof of Lemma~\ref{spreadvector}.
\end{proof}

We also need a counterpart for Theorem \ref{rowAissumeigenspaces}.

\begin{lemma}\label{rowAissumeigenspacesII}
 Let $\mathcal{G}$ be a class of generators of the hyperbolic quadric $\mathcal{P}=\mathcal{Q}^{+}(2d-1,q)$, $d$ even, and let $A'_{0},A'_{1},\dots,A'_{\frac{d}{2}}$ be the matrices of the corresponding association scheme.
 Consider the eigenspace decomposition $\R^{\mathcal{G}}=V'_{0}\perp V'_{1}\perp\dots\perp V'_{\frac{d}{2}}$ related to this association scheme introduced in Remark~\ref{assoschemehyperbolic}. Let $A'$ be the incidence matrix of points of $\mathcal{P}$ and generators in $\mathcal{G}$. Then, $\im(A'^{t})=V'_{0}\perp V'_{1}$.
\end{lemma}
\begin{proof}
 Let $\overline{\mathcal{G}}$ be the other class of generators, and let $\overline{A}_{0},\dots,\overline{A}_{\frac{d}{2}}$ be the  matrices of the association scheme on $\overline{\mathcal{G}}$. Let $A_{0},\dots,A_{d}$ be the matrices of the association scheme on the generators of $\mathcal{P}$, so on $\mathcal{G}\cup\overline{\mathcal{G}}$, and let $V_{0},V_{1},\dots,V_{d}$ be the eigenspaces of this association scheme. Ordering the generators in such a way that we put first the generators of $\mathcal{G}$ and then the generators of $\overline{\mathcal{G}}$, we see that $A_{2i}=\left(\begin{smallmatrix}A'_{i}&0\\0&\overline{A}_{i}\end{smallmatrix}\right)$. We see that if $\bm{w}=\left(\begin{smallmatrix}\bm{w}'\\\overline{\bm{w}}\end{smallmatrix}\right)$ is an eigenvector of $A_{2i}$ with eigenvalue $\lambda$, then $\bm{w}'$ is an eigenvector of $A'_{i}$ with eigenvalue $\lambda$ (and $\overline{\bm{w}}$ is an eigenvector of $\overline{A}_{i}$ with eigenvector $\lambda$). 
 Using Remark \ref{assoschemehyperbolic} we find that if a vector $\bm{w}=\left(\begin{smallmatrix}\bm{w}'\\\overline{\bm{v}}\end{smallmatrix}\right)$ lies in $V_{k}$ for some $k\in\{0,\dots,d\}$, then $\bm{w}'$ lies in $V'_{k}$.
 \par Now, let $A$ be the incidence matrix of points and generators in $\mathcal{P}$. Every column in $A'^{t}$ can be seen as a restriction of the column in $A^{t}$ corresponding to the same point. Let $\bm{v}$ be a vector in $\im(A^{t})$. It is the restriction of a vector $\overline{\bm{v}}\in\im\left(\overline{A}^{t}\right)$ since it is a linear combination of columns of $A^{t}$. By Theorem \ref{rowAissumeigenspaces} we know that $\im\left(\overline{A}^{t}\right)=V_{0}\perp V_{1}$. By the argument in the previous paragraph, we know now that $\bm{v}'\in V'_{0}\perp V'_{1}$. We conclude that $\im(A^{t})\subseteq V'_{0}\perp V'_{1}$.
 \par Now, consider the matrix $A'^{t}A'$. This matrix has the same image as $A'^{t}$. Furthermore, $A'^{t}A'$ lies in the Bose-Mesner algebra of the association scheme, since the entry ${(A'^{t}A')}_{\pi,\pi'}$ is the number of common points of $\pi$ and $\pi'$, and so
 \[
  A'^{t}A'=\sum_{i=0}^{\frac{d}{2}-1}\gs{d-2i}{1}{q}A'_{i}\;.
 \]
 By Lemma~\ref{rowspacesubspacesum}, this implies that $\im(A'^{t})=\im(A'^{t}A')$ is a sum of eigenspaces $V'_{i}$.
 The image of $A'^{t}$ equals thus $V'_{0}$, $V'_{1}$ or $V'_{0}\perp V'_{1}$ since $\im(A'^{t})\subseteq V'_{0}\perp V'_{1}$ and it is non-empty. On the one hand, it is clear that $\bm{j}\in\im(A'^{t})$ because $A'^{t}\bm{j}=\gs{d}{1}{q}\bm{j}$, hence $V'_{0}\subseteq\im(A'^{t})$.
 On the other hand, consider a column $\bm{v}$ of $A'^{t}$.
 Then $\bm{v}-\frac{1}{q^{d-1}+1}\bm{j}$ is a non-zero vector of $V'_{1}$ by the above and the observation that $\bm{v}\notin\left\langle\bm{j}\right\rangle$.
 Since both $\bm{v}$ and $\bm{j}$ are contained in $\im(A'^{t})$, also $\bm{v}-\frac{1}{q^{d-1}+1}\bm{j}$ is contained in $\im(A'^{t})$ and consequently $V'_{1} \subseteq\im(A'^{t})$.
 So we conclude that $\im(A'^{t})=V'_{0}\perp V'_{1}$.
\end{proof}

Now we present the characterisation theorem.

\begin{theorem}\label{characterisationthII}
 Let $\mathcal{G}$ be a class of generators of the hyperbolic quadric $\mathcal{P}=\mathcal{Q}^{+}(2d-1,q)$, $d$ even. Let $\mathcal{L}\subseteq\mathcal{G}$ be a set of generators of $\mathcal{P}$ with characteristic vector $\bm{\chi}$. Let $A'$ be the incidence matrix of points of $\mathcal{P}$ and generators in $\mathcal{G}$, and let $K$ be the disjointness matrix of $\mathcal{G}$. Let $V'_{0}\perp\dots\perp V'_{\frac{d}{2}}$ be the classical decomposition in eigenspaces of the association scheme on the generators in $\mathcal{G}$.
 Denote $\frac{|\mathcal{L}|}{\prod^{d-2}_{i=1}(q^{i}+1)}$ by $x$. The following five statements are equivalent.
 \begin{itemize}
  \item[(i)] For each fixed generator $\pi\in\mathcal{G}$, the number of elements of $\mathcal{L}$ disjoint from $\pi$ equals $(x-{(\bm{\chi})}_{\pi})q^{\binom{d-1}{2}}$.
  \item[(ii)] The vector $\bm{\chi}-\frac{x}{q^{d-1}+1}\bm{j}$ is contained in the eigenspace of $K$ for the eigenvalue $-q^{\binom{d-1}{2}}$.
  \item[(iii)] $\bm{\chi}\in V'_{0}\perp V'_{1}$.
  \item[(iv)] $\bm{\chi}\in\im(A'^{t})$.
  \item[(v)] $\bm{\chi}\in{(\ker(A'))}^{\perp}$.
 \end{itemize}
 If $\mathcal{P}$ admits a spread, then also the next two statements are equivalent to the previous ones.
 \begin{itemize}
  \item[(vi)] $|\mathcal{L}\cap\mathcal{S}|=x$ for every spread $\mathcal{S}$ of $\mathcal{P}$ whose elements are in $\mathcal{G}$.
  \item[(vii)] $|\mathcal{L}\cap\mathcal{S}|=x$ for every spread $\mathcal{S}\in\mathcal{C}$ of $\mathcal{P}$ whose elements are in $\mathcal{G}$, with $\mathcal{C}$ a class of spreads which is a union of some orbits under the action of an automorphism group that acts transitively on the pairs of disjoint generators of $\mathcal{G}$.
 \end{itemize}
\end{theorem}
\begin{proof}
 The proof is analogous to the proofs of Theorems \ref{characterisationth} and \ref{characterisationthI}. For the equivalence between (i), (ii) and (iii) we can rely on Lemma \ref{disjointequiveigenvectorII}, Theorem \ref{mineigenvalue} and Remark \ref{assoschemehyperbolic}. The equivalence between (iii), (iv) and (v) follows from Lemma \ref{rowAissumeigenspacesII}.
 \par From now on we assume that $\mathcal{G}$ admits a spread. We note that by Remark~\ref{spreadstransitive} the existence of a spread of $\mathcal{P}$ implies the existence of a class $\mathcal{C}$ as in statement (vii).
 This shows that statement (vii) is not an empty statement, which is important for the following argument. We show that (v) implies (vi) and that (vii) implies (i). Since (vi) clearly implies (vii), this shows that statements (vi) and (vii) are both equivalent to the statements (i)-(v).
 \par We assume that (vi) is valid. Let $\mathcal{S}\subseteq\mathcal{G}$ be a spread of $\mathcal{P}$, with characteristic vector $\bm{\chi}_{\mathcal{S}}$. By Lemma~\ref{spreadvectorclass} we can find a vector $\bm{w}\in\ker(A')$ such that $\bm{\chi}_{\mathcal{S}}=\bm{w}+\frac{1}{\prod^{d-2}_{i=1}(q^{i}+1)}\bm{j}$. Then
 \[
  |\mathcal{L}\cap\mathcal{S}|=\left\langle\bm{\chi},\bm{\chi}_{\mathcal{S}}\right\rangle=\left\langle \bm{\chi},\bm{w}\right\rangle+\frac{1}{\prod^{d-2}_{i=1}(q^{i}+1)}\left\langle \bm{\chi},\bm{j}\right\rangle=\frac{|\mathcal{L}|}{\prod^{d-2}_{i=1}(q^{i}+1)}=x\;,
 \]
 since $\left\langle \bm{\chi},\bm{w}\right\rangle=0$ by the assumption. This proves (vi).
 \par Finally we assume that (vii) is valid. We denote the class of generators different from $\mathcal{G}$ by $\overline{\mathcal{G}}$. Any automorphism of $\mathcal{P}$ either fixes both $\mathcal{G}$ and $\overline{\mathcal{G}}$, or else switches the sets $\mathcal{G}$ and $\overline{\mathcal{G}}$. Let $H$ be the subgroup of $G$ (of index 2) that fixes both $\mathcal{G}$ and $\overline{\mathcal{G}}$ setwise. The subgroup $H$ acts transitively on the pairs of disjoint generators in $\mathcal{G}$, since $G$ acts transitively on the pairs of disjoint generators of $\mathcal{P}$.
 By Lemma~\ref{spreadintersectiontonumberanddisjoint}(ii) statement (vii) implies (i) since the set of all spreads in $\mathcal{C}$ whose elements are contained in $\mathcal{G}$ is a union of orbits under the action of $H$.
\end{proof}

We use this theorem to define Cameron-Liebler sets also for one class of generators of the hyperbolic quadric of even rank.

\begin{definition}\label{clsetII}
 Let $\mathcal{G}$ be a class of generators of the hyperbolic quadric $\mathcal{Q}^{+}(2d-1,q)$, $d$ even. A generator set $\mathcal{L}\subseteq\mathcal{G}$ that fulfils one of the statements in Theorem~\ref{characterisationthII} (and consequently all of them) is called a \emph{Cameron-Liebler set with parameter} $x=\frac{|\mathcal{L}|}{\prod^{d-2}_{i=1}(q^{i}+1)}$.
\end{definition}

\begin{corollary}
 Let $\mathcal{G}$ be a class of generators of the hyperbolic quadric $\mathcal{Q}^{+}(2d-1,q)$, $d$ even, that admits an $m$-regular system $\mathcal{S}$. If $\mathcal{L}$ is a Cameron-Liebler set with parameter $x$, then $|\mathcal{L}\cap\mathcal{S}|=mx$.
\end{corollary}
\begin{proof}
 Analogous to the proof of Corollary~\ref{CLregularsystemI}, now using Lemma~\ref{spreadvectorclass}.
\end{proof}

Using the definition of Cameron-Liebler sets for a class of generators, we find an alternative definition for the Cameron-Liebler sets for the hyperbolic quadrics of even rank themselves.

\begin{theorem}\label{characterisationthIIbis}
 Let $\mathcal{L}$ be a set of generators of $\mathcal{P}=\mathcal{Q}^{+}(2d-1,q)$, $d$ even, with characteristic vector $\bm{\chi}$, and let $\Omega_{1}$ and $\Omega_{2}$ be the two classes of generators of $\mathcal{P}$. Then, for any $x$, $\mathcal{L}$ is a Cameron-Liebler set with parameter $x$ if and only if both $\mathcal{L}\cap\Omega_{1}$ and $\mathcal{L}\cap\Omega_{2}$ are Cameron-Liebler sets with parameter $x$ of $\Omega_{1}$ and $\Omega_{2}$, respectively.
\end{theorem}
\begin{proof}
 Since $d$ is even, two disjoint generators necessarily belong to the same generator class, see Remark~\ref{hyperclass}. Hence, it follows from Theorem \ref{characterisationth}(i) and Theorem \ref{characterisationthII}(i) that the equivalence in the statement of the theorem is true.
\end{proof}

Our goal in this subsection was to present a characterisation theorem for the Cameron-Liebler sets of the hyperbolic quadrics of even rank using the image of an incidence matrix. Using the previous results we can now do this.

\begin{theorem}
 Let $\mathcal{L}$ be a set of generators of $\mathcal{P}=\mathcal{Q}^{+}(2d-1,q)$, $d$ even, with characteristic vector $\bm{\chi}$, and let $\Omega_{1}$ and $\Omega_{2}$ be the two classes of generators of $\mathcal{P}$. Let $B$ be the incidence matrix whose rows are indexed by the tuples $(P,i)$ with $P$ a point of $\mathcal{P}$ and $i\in\{1,2\}$, whose columns are indexed by the generators of $\mathcal{P}$ and whose entry on the row corresponding to $(P,i)$ and in the column corresponding to $\pi$ equals one if $P$ is contained in $\pi$ and $\pi\in\Omega_{i}$, and zero otherwise. Denote $\frac{|\mathcal{L}|}{\prod^{d-1}_{i=0}(q^{e+i}+1)}$ by $x$. The following three statements are equivalent.
 \begin{itemize}
  \item[(i)] $\bm{\chi}\in\im(B^{t})$.
  \item[(ii)] $\bm{\chi}\in{(\ker(B))}^{\perp}$.
  \item[(iii)] $\mathcal{L}$ is a Cameron-Liebler set.
 \end{itemize}
\end{theorem}
\begin{proof}
 This follows from Theorem \ref{characterisationthIIbis} and Theorem \ref{characterisationthII}(iv).
\end{proof}

\begin{remark}
 So, we have found a characterisation of the Cameron-Liebler sets of the hyperbolic quadric $\mathcal{P}=\mathcal{Q}^{+}(2d-1,q)$, $d$ even, using the image of a matrix. However it is not the point-generator incidence matrix $A$ that we used for the polar spaces of type I. One can easily check that all generator sets whose incidence vector is contained in $\im(A^{t})$ are indeed Cameron-Liebler sets of $\mathcal{P}$.
 \par We now check that the converse is not true: $\im(A^{t})$ does not contain the incidence vectors of all Cameron-Liebler sets of $\mathcal{P}$. Let $P$ and $P'$ be two noncollinear points of $\mathcal{P}$, and let $\mathcal{L}$ be the set of all generators of class $\Omega_{1}$ through $P$ and all generators of class $\Omega_{2}$ through $P'$. Then, $\mathcal{L}$ is a Cameron-Liebler set of $\mathcal{P}$ with parameter 1. Denote the characteristic vector of $\mathcal{L}$ by $\bm{\chi}$. Denote the row of $A$ corresponding to the point $Q$ by $\bm{r}^{t}_{Q}$. We show that $\bm{\chi}\notin\im(A^{t})$. Assume that $\bm{\chi}=\sum_{Q\in\mathcal{P}}a_{Q}\bm{r}_{Q}$ for some $a_{Q}\in\R$. Let $\bm{v}_{P^{\perp},j}$ be the characteristic vector of the set of all generators in $\Omega_{j}$ through $P$, $j=1,2$. Then we know that
 \begin{align*}
  \prod_{i=1}^{d-1}(q^{i}+1) & =\left\langle\chi,\bm{v}_{P^{\perp},1}\right\rangle=a_{P}\prod_{i=0}^{d-1}(q^{i}+1)+\prod_{i=0}^{d-2}(q^{i}+1)\sum_{Q\in P^{\perp}\setminus P}a_{Q}\quad\text{and} \\
  0                          & =\left\langle\chi,\bm{v}_{P^{\perp},2}\right\rangle=a_{P}\prod_{i=0}^{d-1}(q^{i}+1)+\prod_{i=0}^{d-2}(q^{i}+1)\sum_{Q\in P^{\perp}\setminus P}a_{Q}\;,
 \end{align*}
 a contradiction.
\end{remark}

\subsection*{The polar spaces $\mathcal{Q}(4n+2,q)$, all $q$, and $\mathcal{W}(4n+1,q)$, $q$ even}

In this section we present the characterisation theorem for two classes of polar spaces, namely the parabolic quadrics of odd rank and the symplectic polar spaces of odd rank over a field with even characteristic. The main reason for treating these two polar spaces separately from the other polar spaces is that by Lemma~\ref{mineigenvalue} there is an eigenspace different from $V_{1}$ that corresponds to the same eigenvalue of the disjointness matrix as $V_{1}$. To simplify the statements of the theorem we will use the following notation.

\begin{notation}
 The polar spaces $\mathcal{Q}(2d,q)$, with $d$ odd, and $\mathcal{W}(2d-1,q)$, with $d$ odd and $q$ even, are called the polar spaces of \emph{type III}.
\end{notation}

Note that all polar spaces of type III have parameter $e=1$.

\begin{remark}\label{nucleus}
 There is an isomorphism between the parabolic quadrics and the symplectic polar spaces III if $q$ is even. Consider the parabolic quadric $\mathcal{Q}(2d,q)$, $q$ even, and let $N$ be its nucleus. This is the unique point such that any line through $N$ is a tangent line to $\mathcal{Q}(2d,q)$. Projecting $\mathcal{Q}(2d,q)$ from $N$ onto a hyperplane $\alpha$ disjoint to $N$ yields the symplectic polar space $\mathcal{W}(2d-1,q)$, $q$ even. For more information about this link between $\mathcal{W}(2d-1,q)$, $q$ even, and $\mathcal{Q}(2d,q)$, $q$ even, we refer to~\cite[Chapter 22]{ht} and~\cite[Chapter 11]{taylor}.
\end{remark}

By this remark we only have to consider parabolic quadrics in this subsection.

\begin{remark}\label{hyperbolicclass}
 Consider the parabolic quadric $\mathcal{P}=\mathcal{Q}(2d,q)$ with $d$ odd. Any non-tangent hyperplane meets $\mathcal{P}$ in either a hyperbolic quadric $\mathcal{Q}^{+}(2d-1,q)$ or in an elliptic quadric $\mathcal{Q}^{-}(2d-1,q)$.
 Let $\mathcal{S}$ be an $m$-regular system of $\mathcal{P}$ and let $\alpha$ be a hyperplane meeting $\mathcal{Q}(2d,q)$ in a hyperbolic quadric $Q$, necessarily also of rank $d$. We know that $|\mathcal{S}|=m(q^{d}+1)$, and then by counting the points of $Q$, we find that precisely $2m$ of the generators in $\mathcal{S}$ are contained in $Q$.
 We know that the generators of $Q$ can be divided in two classes and since $d$ is odd, any two disjoint generators must belong to different classes. Hence, if $\mathcal{S}$ is a spread ($m=1$), the two generators of $\mathcal{S}$ that are contained in $Q$ necessarily belong to different classes.
 \par From the previous observation it follows that the set of all generators of one class of a hyperbolic quadric $\mathcal{Q}^{+}(2d-1,q)$ embedded in $\mathcal{Q}(2d,q)$, $d$ odd, has precisely one generator in common with every spread.
 This generator set is called a \emph{hyperbolic class}. Note that $\mathcal{P}$ does not necessarily contain a spread.
\end{remark}

Now we present a lemma regarding the image of the matrix $B^{t}$ that we introduced in the previous lemma. This is an analogue of Theorem~\ref{rowAissumeigenspaces} and Lemma~\ref{rowAissumeigenspacesII}.

\begin{lemma}\label{rowBissumeigenspaces}
 Let $\mathcal{P}$ be a polar space of type III of rank $d$ over $\F_{q}$, and let $A_{0},A_{1},\dots,A_{d}$ be the incidence matrices of the corresponding association scheme.
 Consider the eigenspace decomposition $\R^{\Omega}=V_{0}\perp V_{1}\perp\dots\perp V_{d}$ related to this association scheme using the classical order, with $\Omega$ the set of all generators of $\mathcal{P}$. Let $B$ be the incidence matrix of hyperbolic classes and generators of $\mathcal{P}$. Then, $\im(B^{t})=V_{0}\perp V_{1}\perp V_{d}$.
\end{lemma}
\begin{proof}
 Let $Q$ be a hyperbolic quadric $\mathcal{Q}^{+}(2d-1,q)$ embedded in $\mathcal{P}$, and let $\mathcal{G}$ and $\mathcal{L}$ be the two classes of generators of $Q$. Note that $\mathcal{G},\mathcal{L}\subset\Omega$. Let $\bm{v}_{\mathcal{G}}$ and $\bm{v}_{\mathcal{L}}$ be the characteristic vectors of $\mathcal{G}$ and $\mathcal{L}$ respectively. Then both $\bm{v}_{\mathcal{G}}^{t}$ and $\bm{v}_{\mathcal{L}}^{t}$ are rows of $B$.
 \par We now look at the vector $A_{1}\bm{v}_{\mathcal{G}}$. The entry of $A_{1}\bm{v}_{\mathcal{G}}$ corresponding to the generator $\pi$ gives the number of generators in $\mathcal{G}$ which have a $(d-2)$-space in common with $\pi$.
 If $\pi\in\mathcal{G}$, then ${\left(A_{1}\bm{v}_\mathcal{G}\right)}_{\pi}$ equals $0$ since two generators of the same class cannot have a $(d-2)$-space in common.
 If $\pi\in\mathcal{L}$, then ${\left(A_{1}\bm{v}_\mathcal{G}\right)}_{\pi}$ equals $\gs{d}{1}{q}$ since there are $\gs{d}{1}{q}$ different $(d-2)$-spaces in $\pi$, each contained in a unique generator of $\mathcal{G}$.
 If $\pi\notin\mathcal{G}\cup\mathcal{L}$, so $\pi$ is not contained in $Q$, then the entry ${\left(A_{1}\bm{v}_{\mathcal{G}}\right)}_{\pi}$ equals $1$ since the $(d-2)$-space $\pi\cap Q$ is contained in a unique generator belonging to $\mathcal{G}$.
 So,
 \[
  A_{1}\bm{v}_{\mathcal{G}}=\gs{d}{1}{q}\bm{v}_{\mathcal{L}}+(\bm{j}- \bm{v}_{\mathcal{G}}-\bm{v}_{\mathcal{L}})\;.
 \]
 Analogously,
 \[
  A_{1}\bm{v}_{\mathcal{L}}=\gs{d}{1}{q}\bm{v}_{\mathcal{G}}+(\bm{j}- \bm{v}_{\mathcal{G}}-\bm{v}_{\mathcal{L}})\;.
 \]
 This implies that
 \[
  A_{1}\left(\bm{v}_{\mathcal{G}}-\bm{v}_{\mathcal{L}}\right)=-\gs{d}{1}{q}\left(\bm{v}_{\mathcal{G}}-\bm{v}_{\mathcal{L}}\right)\;,
 \]
 and hence that $\bm{v}_{\mathcal{G}}-\bm{v}_{\mathcal{L}}$ is an eigenvector of $A_{1}$ with eigenvalue $-\gs{d}{1}{q}$.
 \par It follows directly from Theorem~\ref{rowAissumeigenspaces}, that the eigenvalues $P_{j,1}$ of the association scheme of a polar space are pairwise different. Hence, we see that $\bm{v}_{\mathcal{G}}-\bm{v}_{\mathcal{L}}\in V_{d}$ since $P_{d,1}=-\gs{d}{1}{q}$.
 \par Let $A$ be the point-generator incidence matrix of $\mathcal{P}$ and let $\bm{v}_{P}^{t}$ be the row of $A$ corresponding to $P$, so $\bm{v}_{P}$ is the characteristic vector of the point-pencil through $P$. Let $\mathcal{H}$ be the set of points in $Q$ and let $\mathcal{H}'$ be the set of points of $\mathcal{P}$ not contained in $Q$. Then,
 \[
  \bm{v}_{\mathcal{G}}+\bm{v}_{\mathcal{L}}=\frac{1}{\gs{d}{1}{q}}\sum_{P\in\mathcal{H}}\bm{v}_{P}-\frac{\gs{d-1}{1}{q}}{\gs{d}{1}{q}q^{d-1}}\sum_{P\in\mathcal{H}'}\bm{v}_{P}\;,
 \]
 since each generator in $\mathcal{G}\cup\mathcal{L}$ contains $\gs{d}{1}{q}$ points of $\mathcal{H}$ and no points of $\mathcal{H}'$, while every generator in $\Omega\setminus(\mathcal{G}\cup\mathcal{L})$ contains $\gs{d-1}{1}{q}$ points of $\mathcal{H}$ and $q^{d-1}$ points of $\mathcal{H}'$.
 Consequently, $\bm{v}_{\mathcal{G}}+\bm{v}_{\mathcal{L}}\in\im(A^{t})$. By Theorem~\ref{rowAissumeigenspaces} we know that $\im(A^{t})=V_{0}\perp V_{1}$.
 \par Since $\bm{v}_{\mathcal{G}}-\bm{v}_{\mathcal{L}}\in V_{d}$ and $\bm{v}_{\mathcal{G}}+\bm{v}_{\mathcal{L}}\in\im(A^{t})$, we know that $\bm{v}_{\mathcal{G}},\bm{v}_{\mathcal{L}}\in\im(A^{t})\oplus V_{d}$.
 The vectors $\bm{v}_{\mathcal{G}}^{t}$ and $\bm{v}_{\mathcal{L}}^{t}$ are arbitrary rows of $B$ since $Q$ can be chosen arbitrarily. It follows that $\im(B^{t})\subseteq\im(A^{t})\perp V_{d}=V_{0}\perp V_{1}\perp V_{d}$.
 \par Now, consider the matrix $B^{t}B$. This matrix has the same image as $B^{t}$. Furthermore,
 \[
  B^{t}B=\sum_{i=0}^{d}q^{d-i}A_{i}
 \]
 since the entry ${(B^{t}B)}_{\pi,\pi'}$ is the number of hyperbolic classes containing both generators $\pi$ and $\pi'$, which equals $q^{k+1}$ if $\dim(\pi\cap\pi')=k$. So, $B^{t}B$ lies in the Bose-Mesner algebra of the association scheme of $\mathcal{P}$.
 By Lemma~\ref{rowspacesubspacesum}, this implies that $\im(B^{t})=\im(B^{t}B)$ is a sum of eigenspaces $V_{i}$.
 Since $\im(B^{t})\subseteq V_{0}\perp V_{1}\perp V_{d}$, the space $\im(B^{t})$ is a sum of some of the eigenspaces $V_{0},V_{1},V_{d}$.
 Let $\mathcal{G}$ and $\mathcal{L}$ be as in the first part of the proof.
 We know on the one hand that $\bm{v}_{\mathcal{G}}-\bm{v}_{\mathcal{L}}$ is a non-zero vector in $V_{d}$ and hence $\im(B^{t})\not\subseteq\im(A^{t})$.
 On the other hand, $\bm{v}_{\mathcal{G}}+\bm{v}_{\mathcal{L}}$ is a non-zero vector in $\im(A^{t})$ and since
 \[
  A_{1}\left(\bm{v}_{\mathcal{G}}+\bm{v}_{\mathcal{L}}\right)=\left(\gs{d}{1}{q}-2\right)\left(\bm{v}_{\mathcal{G}}+\bm{v}_{\mathcal{L}}\right)+2\bm{j}\;,
 \]
 the vector $\bm{v}_{\mathcal{G}}+\bm{v}_{\mathcal{L}}$ is not contained in $V_{0}$ or $V_{1}$. We conclude that $\im(B^{t})=V_{0}\perp V_{1}\perp V_{d}$.
\end{proof}

\begin{theorem}\label{characterisationthIII}
 Let $\mathcal{P}$ be a polar space of type III, and let $B$ be the incidence matrix of hyperbolic classes and generators of $\mathcal{P}$. Let $\mathcal{L}$ be a set of generators of $\mathcal{P}$ with characteristic vector $\bm{\chi}$, and
 denote $\frac{|\mathcal{L}|}{\prod^{d-2}_{i=0}(q^{e+i}+1)}$ by $x$. The following three statements are equivalent.
 \begin{itemize}
  \item[(i)] $\bm{\chi}\in\im(B^{t})$.
  \item[(ii)] $\bm{\chi}\in{(\ker(B))}^{\perp}$.
  \item[(iii)] $\mathcal{L}$ is a Cameron-Liebler set.
 \end{itemize}
\end{theorem}
\begin{proof}
 Statement (i) and statement (ii) are equivalent since $\im(B^{t})={(\ker(B))}^{\perp}$. The equivalence of statements (i) and (iii) follows from Theorem \ref{rowBissumeigenspaces} and Theorem \ref{characterisationth}(iii).
\end{proof}

The main difference between the polar spaces of type I and the polar spaces of type III is that the hyperbolic classes have taken the role of the points (or more precisely the point-pencils); the matrix $B$ has taken the role of the matrix $A$.

\subsection*{The polar spaces $\mathcal{W}(4n+1,q)$, $q$ odd}

In the previous subsections we have proved a characterisation result using the image of a matrix for Cameron-Liebler sets of generators for all polar spaces except the symplectic polar spaces $\mathcal{W}(4n+1,q)$, $q$ odd. For these spaces we were not able to produce a characterisation result similar to Theorems~\ref{characterisationthI},~\ref{characterisationthII} and~\ref{characterisationthIII}. This is the first open problem we want to pose.
\par We know that the image of the transposed point-generator incidence matrix does not correspond to a sum of eigenspaces of the generator disjointness matrix: $\im(A)=V_{0}\perp V_{1}$, but $V_{1}$ is not an eigenspace of the disjointness matrix (being part of the eigenspace $V_{1}\perp V_{d}$). We observed this non-correspondence also in the polar spaces of type III, however for those we could use the incidence matrix of hyperbolic classes and generators as a substitute, which we cannot do here, since no hyperbolic quadrics are embedded in $\mathcal{W}(4n+1,q)$, $q$ odd.
\par Recall from Table~\ref{spreadtable} that the symplectic polar spaces $\mathcal{W}(4n+1,q)$, $q$ odd, admit spreads. So, in Theorem \ref{characterisationth} actually all five statements are present. Hence, we do know some equivalent characterisations for Cameron-Liebler sets of generators for polar spaces $\mathcal{W}(4n+1,q)$, $q$ odd.
\par For the characterisation results in Section \ref{sec:classification} we will call the symplectic polar spaces $\mathcal{W}(4n+1,q)$, with $q$ odd, the polar spaces of \emph{type IV}.

\section{Properties and examples}\label{sec:examples}

The following properties of Cameron-Liebler sets can easily be proved.

\begin{lemma}\label{basicoperations}
 Let $\mathcal{P}$ be a finite classical polar space of type I or III, or one class of generators of $\mathcal{Q}^{+}(2d-1,q)$, $d$ even. If $\mathcal{L}$ and $\mathcal{L}'$ are Cameron-Liebler sets of $\mathcal{P}$ with parameters $x$ and $x'$ respectively, then the following statements are valid.
 \begin{itemize}
  \item[(i)] $0\leq x\leq q^{e+d-1}+1$.
  \item[(ii)] The set of all generators not in $\mathcal{L}$ is a Cameron-Liebler set of $\mathcal{P}$ with parameter $q^{e+d-1}+1-x$.
  \item[(iii)] If $\mathcal{L}\cap\mathcal{L}'=\emptyset$, then $\mathcal{L}\cup\mathcal{L}'$ is a Cameron-Liebler set of $\mathcal{P}$ with parameter $x+x'$.
  \item[(iv)] If $\mathcal{L}'\subset\mathcal{L}$, then $\mathcal{L}\setminus\mathcal{L}'$ is a Cameron-Liebler set of $\mathcal{P}$ with parameter $x-x'$.
 \end{itemize}
\end{lemma}

Now we present some examples of Cameron-Liebler sets.

\begin{example}\label{pointpencil}
 Let $\mathcal{P}$ be a finite classical polar space of type I or III, or one class of generators of $\mathcal{Q}^{+}(2d-1,q)$, $d$ even. The set $\mathcal{L}$ of all generators of $\mathcal{P}$ through a fixed point is a Cameron-Liebler set with parameter $1$.
 Note that this is a point-pencil if $\mathcal{P}$ is a polar space of type I or III.
 \par For polar spaces of type I and a class of generators of $\mathcal{Q}^{+}(2d-1,q)$, $d$ even, this is immediate as the incidence vector of $\mathcal{L}$ is a row in the point-generator incidence matrix.
 For polar spaces of type III this follows from the observation that the image of the transposed point-generator incidence matrix is a subspace of the image of the transposed incidence matrix of hyperbolic classes and generators (see Theorem~\ref{rowAissumeigenspaces} and Lemma~\ref{rowBissumeigenspaces}).
\end{example}

\begin{remark}
 A set of points in a polar space such that each generator contains at most one point of the set is called a \emph{partial ovoid}; it is called an \emph{ovoid} if each generator contains precisely one point of the set. It is immediate that an ovoid of a finite classical polar space of rank $d$ with parameter $e$ over $\F_{q}$ contains $q^{d+e-1}+1$ pairwise non-collinear points. Ovoids and partial ovoids of polar spaces are well-studied and many (non-)existence results are known. There is a close connection between ovoids and spreads. We refer to~\cite[Sections 3 and 4]{dbkm} for more background, details and results.
 \par We note here that it follows from Lemma~\ref{basicoperations}(iii) and Example~\ref{pointpencil} that a finite classical polar space of rank $d$ with parameter $e$ over $\F_{q}$ that admits a partial ovoid of size $\left\lfloor\frac{q^{d+e-1}+1}{2}\right\rfloor$ has a Cameron-Liebler set for all possible parameter values. In particular all polar spaces that admit an ovoid have a Cameron-Liebler set for all possible parameter values.
\end{remark}

\begin{example}\label{embedded}
 Let $\mathcal{P}$ be a finite classical polar space of rank $d$ with parameter $e$ over $\F_{q}$ and let $\mathcal{P}'$ be a finite classical polar space of the same rank with parameter $e-1$ over $\F_{q}$ that is embedded in $\mathcal{P}$, $e\geq1$; all examples of such polar spaces $\mathcal{P}$ and $\mathcal{P}'$ are listed below.
 Let $\mathcal{L}$ be the set of generators contained in $\mathcal{P}'$ and let $\bm{\chi}_{\mathcal{L}}$ be its characteristic vector.
 We denote the point-generator incidence matrix of $\mathcal{P}$ by $A$ and for any point $P\in\mathcal{P}$ we denote the corresponding row of $A$ by $\bm{v}_{P}^{t}$, so $\bm{v}_{P}$ is the characteristic vector of the point-pencil through $P$. Let $\mathcal{H}$ be the set of points in $\mathcal{P}'$ and let $\mathcal{H}'$ be the set of points of $\mathcal{P}$ not contained in $\mathcal{P}'$.
 \par Considering the list of examples we see that all embeddings of $\mathcal{P}'$ in $\mathcal{P}$ arise from a hyperplane intersection (for $\mathcal{W}(2d-1,q)$, $q$ even, we use the viewpoint of Remark \ref{nucleus}). So, any generator of $\mathcal{P}$ that is not contained in $\mathcal{P}'$ has precisely $\gs{d-1}{1}{q}$ points in common with $\mathcal{P}'$. Hence,
 \[
  \bm{\chi}_{\mathcal{L}}=\frac{1}{\gs{d}{1}{q}}\sum_{P\in\mathcal{H}}\bm{v}_{P}-\frac{\gs{d-1}{1}{q}}{\gs{d}{1}{q}q^{d-1}}\sum_{P\in\mathcal{H}'}\bm{v}_{P}\;,
 \]
 since each generator in $\mathcal{L}$ contains $\gs{d}{1}{q}$ points of $\mathcal{H}$ and no points of $\mathcal{H}'$, while every generator not in $\mathcal{L}$ contains $\gs{d-1}{1}{q}$ points of $\mathcal{H}$ and $q^{d-1}$ points of $\mathcal{H}'$.
 Consequently, $\bm{\chi}_{\mathcal{L}}$ is contained in $\im(A^{t})$. So, if $\mathcal{P}$ is a polar space of type I or of type III, then $\mathcal{L}$ is a Cameron-Liebler set with parameter $q^{e-1}+1$.
 \par We give all examples that arise from this construction.
 \begin{itemize}
  \item Consider the elliptic quadric $\mathcal{P}=\mathcal{Q}^{-}(2d+1,q)$, which is of type I. Any non-tangent hyperplane meets $\mathcal{P}$ in a parabolic quadric $\mathcal{Q}(2d,q)$. The set of all generators in such an embedded parabolic quadric is a Cameron-Liebler set of $\mathcal{P}$ with parameter $q+1$.
  \item Consider the parabolic quadric $\mathcal{P}=\mathcal{Q}(2d,q)$, which is of type I if $d$ is even and of type III if $d$ is odd. Any non-tangent hyperplane meets $\mathcal{P}$ in either a hyperbolic quadric $\mathcal{Q}^{+}(2d-1,q)$ or an elliptic quadric $\mathcal{Q}^{-}(2d-1,q)$. The embedded hyperbolic quadrics have the same rank as $\mathcal{P}$. The set of all generators in such an embedded hyperbolic quadric is a Cameron-Liebler set of $\mathcal{P}$ with parameter $2$. Recall that the symplectic variety $\mathcal{W}(2d-1,q)$ is isomorphic to $\mathcal{P}$ if $q$ is even, see Remark~\ref{nucleus}.
  \item Consider the Hermitian polar space $\mathcal{P}=\mathcal{H}(2d,q^{2})$, which is of type I. Any non-tangent hyperplane meets $\mathcal{P}$ in a Hermitian polar space $\mathcal{H}(2d-1,q^{2})$. The set of all generators in such an embedded Hermitian variety is a Cameron-Liebler set of $\mathcal{P}$ with parameter $q+1$.
 \end{itemize}
\end{example}

\begin{example}\label{hypclass}
 Let $\mathcal{P}$ be a finite classical polar space of type III.\@ In Remark~\ref{hyperbolicclass} we introduced the hyperbolic class. Any hyperbolic class of $\mathcal{P}$ is a Cameron-Liebler set with parameter $1$. This is immediate as the incidence vector of a hyperbolic class is a row in the incidence matrix of points and hyperbolic classes.
\end{example}

\begin{example}\label{baseplane}
 Let $\mathcal{P}$ be a finite classical polar space of type III of rank $3$ over $\F_{q}$, hence $\mathcal{P}=\mathcal{W}(5,q)$ or $\mathcal{P}=\mathcal{Q}(6,q)$. Let $\pi$ be a generator (a plane) and let $\mathcal{L}$ be the set of all planes meeting $\pi$ in a line or a plane. It is easy to see that all elements of $\mathcal{L}$ meet each other.
 A plane $\tau\notin\mathcal{L}$ either meets $\pi$ in a point or else has an empty intersection with $\pi$. If $\pi\cap\tau$ is a point, then $\tau$ meets $\pi$ and the $q$ planes through a line of $\pi$ passing through $\pi\cap\tau$. If $\pi\cap\tau=\emptyset$, then through every point of $\tau$ there is a unique plane of $\mathcal{L}$.
 So, in both cases the number of elements of $\mathcal{L}$ meeting $\tau$ equals $q^{2}+q+1$. By Theorem~\ref{characterisationthIII}(iii) $\mathcal{L}$ is a Cameron-Liebler set of $\mathcal{P}$ with parameter $1$.
\end{example}

\begin{example}\label{basesolid}
 Let $\mathcal{P}$ be one class of generators of the quadric $\mathcal{Q}^{+}(7,q)$. Let $\pi$ be a generator ($3$-space) belonging to the other class and let $\mathcal{L}$ be the set of generators of $\mathcal{P}$ meeting $\pi$ in a plane. It is immediate that the elements of $\mathcal{L}$ pairwise meet.
 A generator $\tau\in\mathcal{P}\setminus\mathcal{L}$ meets $\pi$ in a point. If $\tau$ and an element $\rho\in\mathcal{L}$ meet non-trivially, then the intersection $\tau\cap\rho$ must be contained in $\pi$. Hence, $\tau$ meets the $q^{2}+q+1$ elements of $\mathcal{L}$ through a plane of $\pi$ containing $\pi\cap\tau$.
 It follows from Theorem~\ref{characterisationthII}(iii) that $\mathcal{L}$ is a Cameron-Liebler set of $\mathcal{P}$ with parameter $1$. We call this set of generators a \emph{base-solid (with center $\pi$)}.
\end{example}

Recall that we defined the parameter $x$ of a Cameron-Liebler set $\mathcal{L}$ as a fraction in Definitions~\ref{clset} and~\ref{clsetII}. If the polar space admits a spread, each spread has $x$ generators in common with $\mathcal{L}$, and so it is clear that $x$ has to be an integer. Now we prove that $x$ is always an integer.

\begin{theorem}
 Let $\mathcal{P}$ be a finite classical polar space or one class of generators of $\mathcal{Q}^{+}(2d-1,q)$, $d$ even. Let $\mathcal{L}$ be a Cameron-Liebler set of $\mathcal{P}$ with parameter $x$. Then $x\in\N$.
\end{theorem}
\begin{proof}
 We present the proof for a finite classical polar space of type different from II. The proof is (very) similar for a class of generators of $\mathcal{Q}^{+}(2d-1,q)$, $d$ even, and by Theorem \ref{characterisationthIIbis} it also follows for the polar spaces of type II. We denote the rank of $\mathcal{P}$ by $d$ and its parameter by $e$. We denote the underlying finite field by $\F_{q}$.
 By Theorem~\ref{characterisationth}(i) the number of elements of $\mathcal{L}$ that are disjoint to a given generator $\pi$ equals
 \[
  \frac{|\mathcal{L}|\:q^{\binom{d-1}{2}+e(d-1)}}{\prod^{d-2}_{i=0}(q^{i+e}+1)}-\delta q^{\binom{d-1}{2}+e(d-1)}\;,
 \]
 with $\delta$ equal to $1$ if $\pi\in\mathcal{S}$ and $0$ if $\pi\notin\mathcal{S}$. Hence
 \[
  \frac{|\mathcal{L}|\:q^{\binom{d-1}{2}+e(d-1)}}{\prod^{d-2}_{i=0}(q^{i+e}+1)}
 \]
 must be an integer. We know that $\gcd(q^{a},q^{b}+1)=1$ for integers $a$ and $b$ with $b>0$. Since $e\neq0$ ($\mathcal{P}$ is not of type II), we have $\gcd(q^{\binom{d-1}{2}+e(d-1)},\prod^{d-2}_{i=0}(q^{i+e}+1))=1$. So $\prod^{d-2}_{i=0}(q^{i+e}+1)$ must be a divisor of $|\mathcal{L}|$ and hence $x$ is an integer.
\end{proof}

\begin{lemma}\label{distance}
 Let $\mathcal{P}$ be a finite classical polar space of type I of rank $d$ with parameter $e$ and let $\mathcal{L}$ be a Cameron-Liebler set of $\mathcal{P}$ with parameter $x$. The number of elements of $\mathcal{L}$ meeting a generator $\pi$ in a $(d-i-1)$-space, $i=0,\dots,d$, equals
 \[
  \begin{cases}
   \left((x-1)\gs{d-1}{i-1}{q}+q^{i+e-1}\gs{d-1}{i}{q}\right)q^{\binom{i-1}{2}+(i-1)e} & \text{if }\pi\in\mathcal{L}    \\
   x\gs{d-1}{i-1}{q}q^{\binom{i-1}{2}+(i-1)e}                                          & \text{if }\pi\notin\mathcal{L}
  \end{cases}\;.
 \]
\end{lemma}
\begin{proof}
 Let $V_{0},\dots,V_{d}$ be the eigenspaces of the association scheme of $\mathcal{P}$, using the ordering as in Section~\ref{sec:prelim}, and let $A_{0},\dots,A_{d}$ be the matrices of the association scheme. Denote the characteristic vector of the generator set $\mathcal{L}$ by $\bm{\chi}$. By Theorem~\ref{characterisationth}(iii) we know that there is a vector $\bm{v}\in V_{1}$ such that $\bm{\chi}=\bm{v}+\frac{x}{q^{d+e-1}+1}\bm{j}$.
 \par The matrix $A_{i}$ is the incidence matrix of the relation $R_{i}$ which describes whether the intersection dimension of two generators equals $d-i-1$ or not. Hence on the position corresponding to a generator $\pi$ the vector $A_{i}\bm{\chi}$ has the number of generators in $\mathcal{L}$ meeting $\pi$ in a $(d-i-1)$-space. We find
 \begin{align*}
  A_{i}\bm{\chi} & =A_{i}\bm{v}+\frac{x}{q^{d+e-1}+1}A_{i}\bm{j}=P_{1,i}\bm{v}+\frac{x}{q^{d+e-1}+1}P_{0,i}\bm{j}                                                                                                         \\
                 & =\left(\gs{d-1}{i}{q}q^{\binom{i}{2}+ie}-\gs{d-1}{i-1}{q}q^{\binom{i-1}{2}+(i-1)e}\right)\bm{v}+\frac{x}{q^{d+e-1}+1}\gs{d}{i}{q}q^{\binom{i}{2}+ie}\bm{j}                                             \\
                 & =\left(\gs{d-1}{i}{q}q^{\binom{i}{2}+ie}-\gs{d-1}{i-1}{q}q^{\binom{i-1}{2}+(i-1)e}\right)\left(\bm{\chi}-\frac{x}{q^{d+e-1}+1}\bm{j}\right)+\frac{x}{q^{d+e-1}+1}\gs{d}{i}{q}q^{\binom{i}{2}+ie}\bm{j} \\
                 & =\frac{x}{q^{d+e-1}+1}q^{\binom{i-1}{2}+(i-1)e}\left(\gs{d-1}{i-1}{q}-\gs{d-1}{i}{q}q^{i+e-1}+\gs{d}{i}{q}q^{i+e-1}\right)\bm{j}                                                                       \\
                 & \qquad+\left(\gs{d-1}{i}{q}q^{i+e-1}-\gs{d-1}{i-1}{q}\right)q^{\binom{i-1}{2}+(i-1)e}\bm{\chi}                                                                                                         \\
                 & =\left(x\gs{d-1}{i-1}{q}\bm{j}+\left(\gs{d-1}{i}{q}q^{i+e-1}-\gs{d-1}{i-1}{q}\right)\bm{\chi}\right)q^{\binom{i-1}{2}+(i-1)e}\;.
 \end{align*}
 The lemma follows.
\end{proof}


When applying the previous lemma for $i=0$, we need to calculate $\binom{-1}{2}=\frac{(-1)(-2)}{2}=1$. Obviously, we find in the lemma 1 if $\pi\in\mathcal{L}$ and 0 otherwise.
\par For the Cameron-Liebler sets of one class of generators of a hyperbolic quadric $\mathcal{Q}^{+}(2d-1,q)$, $d$ even, we can prove a similar result. Recall that two generators in the same class necessarily meet each other in a $j$-space for some $j\equiv d-1\pmod{2}$ (see Remark~\ref{hyperclass}).

\begin{lemma}\label{distanceII}
 Let $\mathcal{G}$ be a class of generators of the hyperbolic quadric $\mathcal{Q}^{+}(2d-1,q)$, $d$ even, and let $\mathcal{L}$ be a Cameron-Liebler set of $\mathcal{G}$ with parameter $x$. The number of elements of $\mathcal{L}$ meeting a generator $\pi$ in a $(d-2i-1)$-space, $i=0,\dots,\frac{d}{2}$, equals
 \[
  \begin{cases}
   \left((x-1)\gs{d-1}{2i-1}{q}+q^{2i-1}\gs{d-1}{2i}{q}\right)q^{\binom{2i-1}{2}} & \text{if }\pi\in\mathcal{L}    \\
   x\gs{d-1}{2i-1}{q}q^{\binom{2i-1}{2}}                                          & \text{if }\pi\notin\mathcal{L}
  \end{cases}\;.
 \]
\end{lemma}
\begin{proof}
 Let $V'_{0},\dots,V'_{\frac{d}{2}}$ be the eigenspaces of the association scheme of $\mathcal{G}$, using the ordering as in Remark \ref{assoschemehyperbolic}, and let $A'_{0},\dots,A'_{\frac{d}{2}}$ be the matrices of the association scheme. Denote the characteristic vector of the generator set $\mathcal{L}$ by $\bm{\chi}$. By Theorem ~\ref{characterisationthII}(iii) we know that there is a vector $\bm{v}\in V'_{1}$ such that $\bm{\chi}=\bm{v}+\frac{x}{q^{d-1}+1}\bm{j}$.
 \par The matrix $A'_{i}$ is the incidence matrix of the relation $R'_{i}$ which describes whether the intersection dimension of two generators equals $d-2i-1$ or not. Hence on the position corresponding to a generator $\pi$ the vector $A'_{i}\bm{\chi}$ has the number of generators in $\mathcal{L}$ meeting $\pi$ in a $(d-2i-1)$-space. We find
 \begin{align*}
  A'_{i}\bm{\chi} & =A'_{i}\bm{v}+\frac{x}{q^{d-1}+1}A'_{i}\bm{j}=P_{1,2i}\bm{v}+\frac{x}{q^{d-1}+1}P_{0,2i}\bm{j}                                                                                              \\
                  & =\left(\gs{d-1}{2i}{q}q^{\binom{2i}{2}}-\gs{d-1}{2i-1}{q}q^{\binom{2i-1}{2}}\right)\bm{v}+\frac{x}{q^{d-1}+1}\gs{d}{2i}{q}q^{\binom{2i}{2}}\bm{j}                                           \\
                  & =\left(\gs{d-1}{2i}{q}q^{\binom{2i}{2}}-\gs{d-1}{2i-1}{q}q^{\binom{2i-1}{2}}\right)\left(\bm{\chi}-\frac{x}{q^{d-1}+1}\bm{j}\right)+\frac{x}{q^{d-1}+1}\gs{d}{2i}{q}q^{\binom{2i}{2}}\bm{j} \\
                  & =\frac{x}{q^{d-1}+1}q^{\binom{2i-1}{2}}\left(\gs{d-1}{2i-1}{q}-\gs{d-1}{2i}{q}q^{2i-1}+\gs{d}{2i}{q}q^{2i-1}\right)\bm{j}                                                                   \\
                  & \qquad+\left(\gs{d-1}{2i}{q}q^{2i-1}-\gs{d-1}{2i-1}{q}\right)q^{\binom{2i-1}{2}}\bm{\chi}                                                                                                   \\
                  & =\left(x\gs{d-1}{2i-1}{q}\bm{j}+\left(\gs{d-1}{2i}{q}q^{2i-1}-\gs{d-1}{2i-1}{q}\right)\bm{\chi}\right)q^{\binom{2i-1}{2}}\;.
 \end{align*}
 The lemma follows.
\end{proof}

\section{Generalised quadrangles}\label{sec:quadrangles}

In Section~\ref{sec:prelim} we introduced the generalised quadrangles (polar spaces of rank $2$) separately from the polar spaces of rank at least three (see Definition~\ref{defgq}) as for rank 2 there are also non-classical examples. In this section we will introduce the Cameron-Liebler sets for generalised quadrangles as line sets with specific properties. Recall that the dual of a generalised quadrangle arises by interchanging the points and lines.
\par Because of duality, Cameron-Liebler sets for a generalised quadrangle $\mathcal{Q}$, which we want to introduce, can also be described as point sets in the dual generalised quadrangle $\mathcal{Q}^{D}$. However, in this way we find precisely the tight sets of generalised quadrangles, introduced in~\cite{pay87}. We refer to~\cite{blp} for a good introduction. For a generalised quadrangle we use $P^{\perp}$ to denote the set of points collinear with $P$, the point $P$ included.

\begin{definition}\label{tightset}
 A point set $\mathcal{T}$ of a generalised quadrangle of order $(s,t)$ is an \emph{$i$-tight set} if for any point $P$ we have
 \[
  |P^{\perp}\cap\mathcal{T}|=\begin{cases}
   s+i & P\in\mathcal{T}    \\
   i   & P\notin\mathcal{T}
  \end{cases}\;.
 \]
\end{definition}

The characterisation theorem below follows from results for tight sets, see \cite[Section 2]{bds} and \cite[Theorem 2.3.11]{vanhove1}. Of course it can be proved in a way analogous to the proofs of Theorems~\ref{characterisationth},~\ref{characterisationthI} and~\ref{characterisationthII}.

\begin{theorem}\label{characterisationthgq}
 Let $\mathcal{Q}$ be a generalised quadrangle of order $(s,t)$ and let $\mathcal{L}$ be a set of lines of $\mathcal{Q}$ with characteristic vector $\bm{\chi}$.
 Let $A$ be the point-line incidence matrix of $\mathcal{Q}$, and let $K$ be the line disjointness matrix of $\mathcal{Q}$. Denote $\frac{|\mathcal{L}|}{t+1}$ by $x$. The following five statements are equivalent.
 \begin{itemize}
  \item[(i)] $\bm{\chi}\in\im(A^{t})$.
  \item[(ii)] $\bm{\chi}\in{(\ker(A))}^{\perp}$.
  \item[(iii)] For each fixed line $\ell$ of $\mathcal{Q}$, the number of elements of $\mathcal{L}$ disjoint from $\ell$ equals $(x-{(\bm{\chi})}_{\ell})t$.
  \item[(iv)] For each fixed line $\ell$ of $\mathcal{Q}$, the number of elements of $\mathcal{L}$ meeting $\ell$ in a point equals $x+{(\bm{\chi})}_{\ell}(t-1)$.
  \item[(v)] The vector $\bm{\chi}-\frac{x}{st+1}\bm{j}$ is contained in the eigenspace of $K$ for the eigenvalue $-t$.
 \end{itemize}
\end{theorem}

\begin{definition}\label{clclassgq}
 Let $\mathcal{Q}$ be a generalised quadrangle. A line set $\mathcal{L}$ that fulfils one of the statements in Theorem~\ref{characterisationthgq} (and consequently all of them) is called a \emph{Cameron-Liebler set with parameter} $x=\frac{|\mathcal{L}|}{t+1}$.
\end{definition}

\begin{theorem}
 The line set $\mathcal{L}$ in the generalised quadrangle $\mathcal{Q}$ is a Cameron-Liebler set with parameter $x$ if and only if the point set $\mathcal{L}^{D}$ is an $x$-tight set in $\mathcal{Q}^{D}$.
\end{theorem}

An $m$-regular system of a generalised quadrangle is a set of lines such that any point is on precisely $m$ of them. A spread of a generalised quadrangle is a $1$-regular system. The dual of an $m$-regular system is an $m$-ovoid. It is known that an $i$-tight set and an $m$-ovoid always have $mi$ points in common (see \cite[Theorem 4.3]{blp}), so the first part of the following result is immediate. The second part can be proven in a way analogous to the proof of  Theorems~\ref{characterisationth}.

\begin{theorem}
 Let $\mathcal{Q}$ be a generalised quadrangle of order $(s,t)$ and let $\mathcal{L}$ be a Cameron-Liebler set with parameter $x$ of $\mathcal{Q}$. For every $m$-regular system  $\mathcal{S}$ of $\mathcal{Q}$, we have $|\mathcal{L}\cap\mathcal{S}|=mx$.
 \par On the other hand, if $\mathcal{Q}$ admits a spread and has an automorphism group $G$ that acts transitively on the pairs of disjoint lines of $\mathcal{P}$ and $\mathcal{L}$ is a set of lines such that $|\mathcal{L}\cap\mathcal{S}|=x$ for every spread $\mathcal{S}\in\mathcal{C}$ of $\mathcal{Q}$, with $\mathcal{C}$ a class of spreads which is a union of orbits under the action of $G$, then $\mathcal{L}$ is a Cameron-Liebler set with parameter $x$.
\end{theorem}

We present a classification result on tight sets of generalised quadrangles. This yields an equivalent result for Cameron-Liebler sets of generalised quadrangles, but we will state it here in the context of tight sets. We first recall a known result.

\begin{theorem}[{\cite[II.5]{pay87}}]\label{onaline}
 Let $\mathcal{Q}$ be a generalised quadrangle of order $(s,t)$. If $\mathcal{T}$ is an $i$-tight set, with $i\leq s$, then a line of $\mathcal{Q}$ is contained in $\mathcal{T}$ or contains at most $i$ points of $\mathcal{T}$.
\end{theorem}

\begin{theorem}\label{classfors>t}
 Let $\mathcal{Q}$ be a generalised quadrangle of order $(s,t)$, with $s\geq t$. If $\mathcal{T}$ is a tight set of $\mathcal{Q}$ with parameter $x\leq\frac{s}{t}+1$, then $\mathcal{T}$ is the union of the point set of $x$ pairwise disjoint lines, or $x=\frac{s}{t}+1$ and the points of $\mathcal{T}$ form a subquadrangle of $\mathcal{Q}$ of order $\left(\frac{s}{t},t\right)$.
\end{theorem}
\begin{proof}
 We assume that $\mathcal{T}$ is non-empty and does not contain a line, hence on every line there is a point not in $\mathcal{T}$. We consider the point $P\in\mathcal{T}$.  We know that there are $x+s-1$ points in $\mathcal{T}$ collinear with $P$, not including $P$. By Theorem~\ref{onaline} there are at most $x-1$ points of $\mathcal{T}\setminus\{P\}$ on a line through $P$ since on this line there must be a point not in $\mathcal{T}$. Consequently,
 \begin{align}\label{eqtight}
  x+s-1\leq(x-1)(t+1)\quad & \Leftrightarrow\quad x\geq\frac{s}{t}+1\;.
 \end{align}
 \par If $x<\frac{s}{t}+1$, we find a contradiction. So, if $\mathcal{T}$ is a non-empty tight set with parameter $x<\frac{s}{t}+1$, then it contains the point set $\mathcal{T}_{\ell}$ of a line $\ell$. It follows that $\mathcal{T}\setminus\mathcal{T}_{\ell}$ is a tight set with parameter $x-1$, and moreover on each line meeting $\ell$ there is a point not in $\mathcal{T}$.
 We can proceed using induction. We find that $\mathcal{T}$ is the union of $x$ disjoint lines.
 \par If $x=\frac{s}{t}+1$, we can have equality in~\eqref{eqtight} if and only if on each line through a point of $\mathcal{T}$ there are precisely $x$ points of $\mathcal{T}$, equivalently on each line there are $0$ or $x$ points of $\mathcal{T}$. If $\mathcal{T}$ is a tight set with parameter $x=\frac{s}{t}+1$, then either it contains the point set of a line or else it does not.
 In the former case $\mathcal{T}$ is the union of a line and a tight set with parameter $\frac{s}{t}$, hence the union of $\frac{s}{t}+1$ lines. In the latter case we denote the set of lines containing $\frac{s}{t}+1$ points of $\mathcal{T}$ by $\mathcal{L}$.
 Now the lines of $\mathcal{L}$ and the points of $\mathcal{T}$ do form a subquadrangle of $\mathcal{Q}$ of order $\left(\frac{s}{t},t\right)$ since
 \begin{itemize}
  \item each point of $\mathcal{T}$ lies on $t+1$ lines of $\mathcal{L}$,
  \item on every line of $\mathcal{L}$ there are $\frac{s}{t}+1$ points of $\mathcal{T}$,
  \item any two lines of $\mathcal{L}$ that meet, meet in a point of $\mathcal{T}$.\qedhere
 \end{itemize}
\end{proof}

\begin{remark}
 For the generalised quadrangles of order $(s,s)$ this coincides with the classical result from~\cite[II.3 and II.4]{pay87}, which classified the $i$-tight sets with $i=1,2$ for all generalised quadrangles.
\end{remark}

By the previous remark we will focus on generalised quadrangles of order $(s,t)$ with $s>t$ in the discussion of Theorem~\ref{classfors>t}.

\begin{remark}
 Theorem~\ref{classfors>t} does not improve the known classification results for the classical generalised quadrangles with $s>t$, except for $\mathcal{H}(3,q^{2})$, with $q=2,3,4$. See Table~\ref{overviewtightsets} for an overview. For the non-classical generalised quadrangle ${\mathcal{H}(4,q^{2})}^{D}$ of order $(q^{3},q^{2})$, Theorem~\ref{classfors>t} gives the best known result.
 \par The non-classical generalised quadrangle $T_{3}(O)$, with $O$ an ovoid of $\PG(3,q)$, has order $(q,q^{2})$. For its construction we refer to~\cite[3.1.2]{pt}.
 It is known that $T_{3}(O)$ is isomorphic to $\mathcal{Q}^{-}(5,q)$ if and only if $O$ is an elliptic quadric (see~\cite[3.2.4]{pt}), so ${T_{3}(O)}^{D}$ is isomorphic to $\mathcal{H}(3,q^{2})$ if and only if $O$ is an elliptic quadric. It is a classical result that all ovoids in $\PG(3,q)$, $q$ odd, are elliptic quadrics (see~\cite{bar,pan}), but in $\PG(3,q)$, $q=2^{h}$ and $h\geq3$ odd, ovoids that are not projectively equivalent to elliptic quadrics are known (see \cite{tit}).
 For the generalised quadrangles ${T_{3}(O)}^{D}$, with $O$ an ovoid that is not an elliptic quadric (hence not isomorphic to $\mathcal{H}(3,q^{2})$), Theorem~\ref{classfors>t} gives the best classification result for tight sets.
\end{remark}

In Table~\ref{overviewtightsets} we give an overview of the classification results for tight sets. Recall that the complement of an $x$-tight set in a generalised quadrangle of order $(s,t)$ is an $(st-x+1)$-tight set, so we restrict to the results for $x\leq\frac{st+1}{2}$.
We only include the generalised quadrangles for which there is a classification result for the tight sets with parameter at least 3 (as for generalised quadrangles the classification of $x$-tight sets with $x\leq 2$ is given in~\cite[II.3 and II.4]{pay87}).
Recall that $\mathcal{Q}(4,q)\cong{\mathcal{W}(3,q)}^{D}$, that $\mathcal{Q}(4,q)\cong{\mathcal{Q}(4,q)}^{D}$ iff $q$ is even, and that $\mathcal{Q}^{-}(5,q)\cong{\mathcal{H}(3,q^{2})}^{D}$ (see~\cite[Section 3.2]{pt}).

\begin{table*}[ht]
 \centering
 \begin{tabular}{|c|ccc|}
  \hline
  $x$-tight set                & classification                         & condition    & Reference                 \\ \hline\hline
  grid                         & complete                               &              & trivial                   \\ \hline
  dual grid                    & complete                               &              & trivial                   \\ \hline
  $\mathcal{H}(3,q^{2})$       & $x<\frac{\sqrt[3]{q^{4}}-1}{2}$        & $q\geq9$ odd & \cite[Theorem 2.17]{ns}   \\
                               & $x<\frac{\sqrt[4]{q^{5}}}{\sqrt{2}}+1$ & $q>4$        & \cite[Theorem 3.9]{dghs}  \\
                               & $x<\epsilon_{q^{2}}=q+1$               &              & \cite[Theorem 15]{bklp}   \\
                               & $x\leq q+1$                            &              & Theorem~\ref{classfors>t} \\ \hline
  $\mathcal{Q}^{-}(5,2)$       & complete                               &              & \cite[IV]{pay87}          \\ \hline
  $\mathcal{Q}(4,q)$, $q$ even & $x<\frac{\sqrt[8]{q^{5}}}{\sqrt{2}}+1$ & $q$ square   & \cite[Theorem 3.12]{dghs} \\
                               & $x<\epsilon_{q}$                       &              & \cite[Theorem 15]{bklp}   \\
                               & complete                               & $q=2$        & \cite[IV]{pay87}          \\
                               & $x\leq 3$                              & $q=4$        & \cite[VII]{pay87}         \\  \hline
  $\mathcal{Q}(4,3)$           & complete                               &              & \cite[V and VI]{pay87}    \\ \hline
  $\mathcal{W}(3,q)$, $q$ odd  & $x<\frac{\sqrt[3]{q^{2}}-1}{2}$        & $q\geq81$    & \cite[Corollary 2.24]{ns} \\
                               & $x<\frac{\sqrt[8]{q^{5}}}{\sqrt{2}}+1$ & $q$ square   & \cite[Theorem 3.12]{dghs} \\
                               & $x<\epsilon_{q}$                       &              & \cite[Theorem 15]{bklp}   \\
                               & complete                               & $q=3$        & \cite[V and VI]{pay87}    \\  \hline
  ${\mathcal{H}(4,q^{2})}^{D}$ & $x\leq q+1$                            &              & Theorem~\ref{classfors>t} \\ \hline
  ${T_{3}(O)}^{D}$             & $x\leq q+1$                            &              & Theorem~\ref{classfors>t} \\ \hline
 \end{tabular}
 \caption{Overview of the results on $x$-tight sets in generalised quadrangles, equivalently Cameron-Liebler sets in the dual generalised quadrangle. Here $O$ is an ovoid in $\PG(3,q)$, $q$ even, that is not an elliptic quadric. The value $\epsilon_{q}$ is such that $q+\epsilon_{q}$ is the size of the smallest blocking set in $\PG(2,q)$ not containing a line.}\label{overviewtightsets}
\end{table*}

\section{Classification results}\label{sec:classification}

In this section we present some classification results for Cameron-Liebler sets of generators in polar spaces. We obtain a strong result in the case of polar spaces of type I with $e\geq1$. For the hyperbolic quadrics of even rank we will focus on classification results for one class of generators because of Theorem \ref{characterisationthIIbis}.
\par Some arguments in the proofs of this section rely on results about Erd\H{o}s-Ko-Rado sets of generators in polar spaces. The Erd\H{o}s-Ko-Rado problem is a well-known problem in combinatorics, which finds its origins in~\cite{ekr}, but which has been studied in many contexts.
For background on the Erd\H{o}s-Ko-Rado problem for generators of polar spaces we refer to~\cite{psv,vanhove1}. As we will see below it is closely related to the Cameron-Liebler problem.

\begin{definition}
 An \emph{Erd\H{o}s-Ko-Rado} set of generators of a polar space is a set of pairwise not disjoint generators.
\end{definition}

The following lemmata will be useful in the proof of the first classification theorem.

\begin{lemma}[{\cite[Lemma 4]{psv}}]\label{psvlemma}
 Let $\pi_{1}$, $\pi_{2}$ and $\pi_{3}$ be pairwise non-disjoint generators in a classical polar space. The intersections $\pi_{1}\cap\pi_{2}$ and $\pi_{1}\cap\pi_{3}$ cannot be complementary subspaces of $\pi_{1}$.
\end{lemma}

\begin{lemma}[{\cite[Corollaries 5 and 12]{kms}}]\label{kmslemma}
 Let $\pi$ and $\pi'$ be two generators of a polar space $\mathcal{P}$ that meet in a $v$-space.
 \begin{itemize}
  \item If $\mathcal{P}=\mathcal{Q}^{+}(2d+1,q)$ and $v\equiv d\pmod{2}$, the number of generators disjoint to both $\pi$ and $\pi'$ equals
        \[
         q^{\frac{(d+v+2)(d+v)}{4}-\binom{v+1}{2}}\prod^{(d-v)/2}_{i=1}(q^{2i-1}-1)\;.
        \]
  \item If $\mathcal{P}=\mathcal{H}(2d+1,q^{2})$, the number of generators disjoint to both $\pi$ and $\pi'$ equals
        \[
         q^{{(d+1)}^2-\binom{d-v+1}{2}}\prod^{d-v}_{i=1}(q^{i}+{(-1)}^{i})\;.
        \]
 \end{itemize}
\end{lemma}
We present a first classification theorem, for Cameron-Liebler sets with parameter 1.

\begin{theorem}\label{par1}
 Let $\mathcal{P}$ be a finite classical polar space of type I, III or IV, or one class of generators of $\mathcal{Q}^{+}(2d-1,q)$, $d$ even, and let $\mathcal{L}$ be a Cameron-Liebler set of $\mathcal{P}$ with parameter $1$.
 \begin{itemize}
  \item If $\mathcal{P}$ is a polar space of type I or IV, then $\mathcal{L}$ is a point-pencil.
  \item If $\mathcal{P}$ is one class of generators of $\mathcal{Q}^{+}(2d-1,q)$, $d$ even, then $\mathcal{L}$ is a point-pencil, or $d=4$ and $\mathcal{L}$ is a base-solid (described in Example~\ref{basesolid}).
  \item If $\mathcal{P}$ is a polar space of type III, then $\mathcal{L}$ is a point-pencil or a hyperbolic class, or $d=3$ and $\mathcal{L}$ is given by Example~\ref{baseplane}.
 \end{itemize}
\end{theorem}
\begin{proof}
 It follows immediately from Theorems~\ref{characterisationth}(i) and ~\ref{characterisationthII}(i) that the elements of $\mathcal{L}$ pairwise meet, hence $\mathcal{L}$ is an Erd\H{o}s-Ko-Rado set (which is not necessarily maximal). Moreover, we know that $|\mathcal{L}|$ is the size of a point-pencil of $\mathcal{P}$ by Definitions \ref{clset} and \ref{clsetII}.
 \par If $\mathcal{P}$ is a generalised quadrangle (or one class of lines of $\mathcal{Q}^{+}(3,q)$), then any Erd\H{o}s-Ko-Rado set is a subset of a point-pencil. The theorem is thus immediate in this case. Now, we assume that $\mathcal{P}$ is an elliptic quadric $\mathcal{Q}^{-}(2d+1,q)$, $d\geq3$, a parabolic quadric $\mathcal{Q}(2d,q)$, $d\geq3$, one class of generators of a hyperbolic quadric $\mathcal{Q}^{+}(4n-1,q)$, $n\geq2$, a symplectic variety $\mathcal{W}(2d-1,q)$, a Hermitian variety $\mathcal{H}(4n-1,q)$, $n\geq2$ and $q$ a square, or a Hermitian variety $\mathcal{H}(2d,q)$, $d\geq3$ and $q$ a square. In these cases the theorem follows immediately from~\cite[Theorems 15, 21--24]{psv}, summarized in~\cite[Section 9]{psv}. Note that we know by Examples~\ref{pointpencil},~\ref{hypclass},~\ref{baseplane} and~\ref{basesolid} that the described Erd\H{o}s-Ko-Rado sets are indeed Cameron-Liebler sets with parameter 1.
 \par Now we assume that $\mathcal{P}$ is either a hyperbolic quadric $\mathcal{Q}^{+}(4n+1,q)$ or a Hermitian variety $\mathcal{H}(4n+1,q)$, $q$ a square. Note that $\mathcal{P}$ is a classical polar space of type I with rank $d=2n+1$.
 First we prove the following claim: for any generator $\pi$ in $\mathcal{L}$ the subset of $\mathcal{L}$ of generators meeting $\pi$ in precisely a point or in precisely a $(d-2)$-space is the set of all generators through a fixed point $P_{\pi}$ that meet the given generator in precisely a point, or the set of all generators through $P_{\pi}$ that meet the given generator in precisely a $(d-2)$-space, respectively.
 \par Let $\pi$ be an element of $\mathcal{L}$. Denote the subset of $\mathcal{L}$ of generators meeting $\pi$ in a $(d-2)$-space by $\mathcal{L}'$ and the subset of $\mathcal{L}$ of generators meeting $\pi$ in precisely a point by $\mathcal{L}''$.
 By Lemma~\ref{distance} we know that $\mathcal{L}''$ is non-empty. Let $\pi'$ be a generator of $\mathcal{L}''$ meeting $\pi$ in a point $P$. Then, by Lemma~\ref{psvlemma}, any generator in $\mathcal{L}'$ and $\pi'$ meet non-trivially in $\pi$. Hence, all elements of $\mathcal{L}'$ contain $P$. By Lemma~\ref{distance} we know that $|\mathcal{L}'|=\gs{d-1}{1}{q}q^{e}$.
 So, $\mathcal{L}'$ is the set of all generators of $\mathcal{P}$ that meet $\pi$ in a $(d-2)$-space through $P$ by Lemma~\ref{skewgenerators}. Now, again by Lemma~\ref{psvlemma} all elements of $\mathcal{L}''$ pass through $P$, and by Lemmas~\ref{skewgenerators} and~\ref{distance} we know that any generator meeting $\pi$ in precisely $P$ is contained in $\mathcal{L}''\subset\mathcal{L}$. This finishes the proof of the claim with $P_{\pi}=P$.
 \par We consider the case in which $\mathcal{P}$ is a Hermitian variety $\mathcal{H}=\mathcal{H}(4n+1,q)$, $q$ a square. Let $\pi$ be an element of $\mathcal{L}$. We know that $\mathcal{L}$ contains all generators through $P_{\pi}$ meeting $\pi$ in precisely a point.
 Assume that $\sigma$ is a generator in $\mathcal{L}$ that does not contain $P_{\pi}$. Let $\alpha$ be a $(4n-1)$-space in the tangent space $T_{P_{\pi}}$ to $\mathcal{H}(4n+1,q)$ at $P_{\pi}$, such that $P_{\pi}\notin\alpha$ and such that $\alpha\supset\sigma\cap T_{P_{\pi}}=\overline{\sigma}$.
 Denote the Hermitian variety $\alpha\cap\mathcal{H}$ by $\mathcal{H}'$. Note that $\overline{\sigma}$ is a generator of $\mathcal{H}'$. Furthermore, any generator $\tau$ of $\mathcal{H}$ through $P_{\pi}$ corresponds to the generator $\alpha\cap\tau$ of $\mathcal{H}'$; let $\overline{\pi}$ be the generator of $\mathcal{H}'$ corresponding to $\pi$. By Lemma~\ref{kmslemma} we know that there is a generator $\overline{\sigma}'$ of $\mathcal{H}'$ disjoint to both $\overline{\pi}$ and $\overline{\sigma}$.
 The subspace $\sigma'=\left\langle\overline{\sigma}',P_{\pi}\right\rangle$ is a generator of $\mathcal{H}$ meeting $\pi$ in precisely the point $P_{\pi}$ and disjoint to $\sigma$. By the above claim $\sigma'$ is a generator in $\mathcal{L}$, however $\sigma'$ is disjoint to $\sigma$, a contradiction as we assumed $\sigma\in\mathcal{L}$. We conclude that all generators in $\mathcal{L}$ must contain the point $P_{\pi}$, hence $\mathcal{L}$ is a point-pencil.
 \par Finally, we consider the case in which $\mathcal{P}$ is a hyperbolic quadric $\mathcal{Q}=\mathcal{Q}^{+}(4n+1,q)$. By Remark~\ref{hypclass} we know that there are two classes of generators on $\mathcal{Q}$, which we will denote by $\Omega_{1}$ and $\Omega_{2}$. Let $\pi$ be an element of $\mathcal{L}$. Without loss of generality we can assume $\pi\in\Omega_{1}$.
 We know that $\mathcal{L}$ contains all generators through $P_{\pi}$ meeting $\pi$ in precisely a point; all generators meeting $\pi$ in precisely a point belong to $\Omega_{1}$. Let $\pi'\in\mathcal{L}$ be a generator such that $\pi\cap\pi'$ is the point $P_{\pi}$. Then $P_{\pi'}=P_{\pi}$ by the above claim since $\pi$ is a generator in $\mathcal{L}$ meeting $\pi'$ in precisely a point.
 \par Now, let $\alpha$ be a $(4n-1)$-space in the tangent space $T_{P_{\pi}}$ to $\mathcal{Q}^{+}(4n+1,q)$ at $P_{\pi}$, such that $P_{\pi}\notin\alpha$. Denote the hyperbolic quadric $\alpha\cap\mathcal{Q}$ by $\mathcal{Q}'$.
 Any generator $\tau$ of $\mathcal{Q}$ through $P_{\pi}$ corresponds to the generator $\alpha\cap\tau$ of $\mathcal{Q}'$; let $\overline{\pi}$ be the generator of $\mathcal{H}'$ corresponding to $\pi$. Consider a generator $\sigma\in\Omega_{1}$ containing $P_{\pi}$ and let $\overline{\sigma}$ be its corresponding generator in $\mathcal{Q}'$.
 Since $\pi$ and $\sigma$ belong to the same generator class of $\mathcal{Q}$, the generators $\overline{\pi}$ and $\overline{\sigma}$ belong to the same class of $\mathcal{Q}'$. Hence, by Lemma~\ref{kmslemma} we know there is a generator $\overline{\sigma}'$ of $\mathcal{Q}'$ disjoint to both $\overline{\pi}$ and $\overline{\sigma}$.
 The subspace $\sigma'=\left\langle\overline{\sigma}',P_{\pi}\right\rangle$ is a generator of $\mathcal{Q}$ meeting $\pi$ in precisely the point $P_{\pi}$, hence $\sigma'\in\mathcal{L}$.
 We observed before that $P_{\sigma'}=P_{\pi}$. So, since $\sigma'$ meets $\sigma$ in precisely $P_{\sigma'}$ (because $\overline{\sigma}$ and $\overline{\sigma}'$ are disjoint), $\sigma$ is contained in $\mathcal{L}$.
 \par So far, we conclude that the set $\mathcal{L}_{1}$ of all generators of the class $\Omega_{1}$ through $P_{\pi}$ is contained in $\mathcal{L}$. By the above claim we also know that $\mathcal{L}$ contains generators of the class $\Omega_{2}$, namely those meeting $\pi$ in a $(2n-1)$-space containing $P_{\pi}$. Let $\rho$ be such a generator of $\Omega_{2}$.
 Analogously, we can prove that the set $\mathcal{L}_{2}$ of all generators of the class $\Omega_{2}$ through $P_{\rho}$ is contained in $\mathcal{L}$. It is immediate that $|\mathcal{L}|=|\mathcal{L}_{1}|+|\mathcal{L}_{2}|$ as $|\mathcal{L}|$ is the size of a point-pencil.
 Consequently, $\mathcal{L}=\mathcal{L}_{1}\cup\mathcal{L}_{2}$. If $P_{\pi}\neq P_{\rho}$, then $\mathcal{L}$ contains clearly pairwise disjoint generators, contradicting that $\mathcal{L}$ is an Erd\H{o}s-Ko-Rado set. We conclude that $P_{\pi}=P_{\rho}$ and that $\mathcal{L}$ is a point-pencil.
\end{proof}

For the polar spaces of type I with $e\geq1$ we extend the previous classification result. We classify all Cameron-Liebler sets with parameter at most $q^{e-1}+1$. First we prove a lemma.

\begin{lemma}\label{zj}
 Let $\mathcal{P}$ be a finite classical polar space of type I of rank $d$ and with parameter $e$, and let $\mathcal{L}$ be a Cameron-Liebler set of $\mathcal{P}$ with parameter $x$. Let $\pi$ be a generator not in $\mathcal{L}$ and let $P$ be a point in $\pi$.
 Now, let $z_{j}$ be the number of elements of $\mathcal{L}$ meeting $\pi$ in a $j$-space through $P$, $0\leq j\leq d-2$. Then, $z_{d-j-2}=z_{d-2}\gs{d-2}{j}{q}q^{\binom{j}{2}+je}$, for $j=0,\dots,d-2$.
\end{lemma}
\begin{proof}
 We prove this lemma using induction on $j$. It is trivially fulfilled for $j=0$, so we prove the induction step: assuming the lemma is true for all values up to (and including) $j-1$, we prove it for $j$.
 \par By Lemma~\ref{distance} we know that there are $x\gs{d-1}{j}{q}q^{\binom{j}{2}+je}$ elements of $\mathcal{L}$ meeting $\pi$ in a $(d-j-2)$-space, $z_{d-j-2}$ of them through $P$. We want to count the elements of $\mathcal{L}$ meeting $\pi$ in a $(d-j-2)$-space not through $P$.
 Denote the number of $(d-2)$-spaces in $\pi$ not through $P$ that are contained in $i$ elements of $\mathcal{L}$ by $t_{i}$. Then it is immediate that $\sum_{i\geq0}t_{i}=q^{d-1}$ and that $\sum_{i\geq0}it_{i}=x-z_{d-2}$. We perform a double counting of the following set:
 \[
  \mathcal{T}=\{(\sigma,\alpha)\mid\sigma\in\mathcal{L},\dim(\sigma\cap\pi)=d-j-2,\dim(\alpha)=d-2,\alpha\subset\pi,P\notin\alpha,\sigma\cap\pi\subseteq\alpha\}\;.
 \]
 For a generator $\sigma\in\mathcal{L}$ with $P\notin\sigma$ and $\dim(\sigma\cap\pi)=d-j-2$ we can find $q^{j}$ different $(d-2)$-spaces in $\pi$ passing through $\sigma\cap\pi$ and not containing $P$. Now, we look at a $(d-2)$-space $\alpha\subset\pi$ not containing $P$. Let $\tau$ be a generator of $\mathcal{P}$ meeting $\pi$ in $\alpha$. For a generator $\rho\in\mathcal{L}$ meeting $\tau$ in a $(d-j-1)$-space, there are three possibilities:
 \begin{itemize}
  \item $\rho\cap\pi$ is a $(d-j-2)$-space contained in $\tau\cap\pi=\alpha$ ($(\rho,\tau,\alpha)$ is of sort 1);
  \item $\rho\cap\pi$ is a $(d-j-1)$-space contained in $\alpha$ ($(\rho,\tau,\alpha)$ is of sort 2);
  \item $\rho\cap\pi$ is a $(d-j)$-space that meets $\alpha$ in a $(d-j-1)$-space ($(\rho,\tau,\alpha)$ is of sort 3).
 \end{itemize}
 By Lemma~\ref{distance} we know the number of elements of $\mathcal{L}$ meeting $\tau$ in a $(d-j-1)$-space, hence the number of tuples $(\rho,\tau,\alpha)$ with $\alpha$ a $(d-2)$-space in $\pi$ not through $P$, $\tau$ a generator meeting $\pi$ in $\alpha$ and $\rho$ an element of $\mathcal{L}$ meeting $\tau$ in a $(d-j-1)$-space:
 \begin{align*}
   & \sum_{i\geq0}t_{i}\left[i\left((x-1)\gs{d-1}{j-1}{q}+q^{j+e-1}\gs{d-1}{j}{q}\right)q^{\binom{j-1}{2}+(j-1)e}+(q^{e}-i)x\gs{d-1}{j-1}{q}q^{\binom{j-1}{2}+(j-1)e}\right] \\
   & =q^{\binom{j-1}{2}+(j-1)e}\left(xq^{e}\gs{d-1}{j-1}{q}\sum_{i\geq0}t_{i}+\left(q^{j+e-1}\gs{d-1}{j}{q}-\gs{d-1}{j-1}{q}\right)\sum_{i\geq0}it_{i}\right)                \\
   & =q^{\binom{j-1}{2}+(j-1)e}\left(xq^{d+e-1}\gs{d-1}{j-1}{q}+\left(q^{j+e-1}\gs{d-1}{j}{q}-\gs{d-1}{j-1}{q}\right)(x-z_{d-2})\right)\;.
 \end{align*}
 We denote the set of these tuples $(\rho,\tau,\alpha)$ by $\mathcal{T}'$. Note that for any tuple $(\rho,\alpha)\in\mathcal{T}$ the generator $\tau\supset\alpha$, meeting $\rho$ in a $(d-j-1)$-space is uniquely determined; hence for any tuple $(\rho,\alpha)\in\mathcal{T}$ there is a unique tuple $(\rho,\tau,\alpha)\in\mathcal{T}'$ of sort 1. If $(\rho,\tau,\alpha)\in\mathcal{T}'$ is of sort 1, then $(\rho,\alpha)\in\mathcal{T}$.
 \par If $(\rho,\tau,\alpha)\in\mathcal{T}'$ is of sort 2, then $(\rho,\alpha)\notin\mathcal{T}$ and there are $q^{e}-1$ choices for $\tau$ given $\rho$ and $\alpha$. The number of tuples $(\rho,\tau,\alpha)\in\mathcal{T}'$ of sort $2$ is
 \[
  \left(x\gs{d-1}{j-1}{q}q^{\binom{j-1}{2}+(j-1)e}-z_{d-j-1}\right)q^{j-1}(q^{e}-1)\;,
 \]
 since there are $x\gs{d-1}{j-1}{q}q^{\binom{j-1}{2}+(j-1)e}-z_{d-j-1}$ elements of $\mathcal{L}$ meeting $\pi$ in a $(d-j-1)$-space not through $P$ by Lemma~\ref{distance} and there are $q^{j-1}$ different $(d-2)$-spaces in $\pi$ not through $P$, containing a given $(d-j-1)$-space.
 \par If $(\rho,\tau,\alpha)\in\mathcal{T}'$ is of sort 3, then $(\rho,\alpha)\notin\mathcal{T}$. Note that we may assume here that $j\geq2$ since $\rho\neq\pi$. To count the number of these tuples $(\rho,\tau,\alpha)\in\mathcal{T}'$ of sort 3 we must distinguish between the generators $\rho$ that meet $P$ and the generators $\rho$ that do not.
 The total number of generators $\rho$ such that $\dim(\rho\cap\pi)=d-j$ equals $x\gs{d-1}{j-2}{q}q^{\binom{j-2}{2}+(j-2)e}$ by Lemma~\ref{distance}, of which $z_{d-j}$ pass through $P$.
 If $\rho$ passes through $P$, there are $q^{d-1}$ choices for $\alpha$ and for a given $\alpha$ and $\rho$ precisely $q^{e}$ choices for $\tau$. If $\rho$ does not pass through $P$, there are $q^{d-1}-q^{j-2}$ choices for $\alpha$, and for a given $\alpha$ and $\rho$ there are precisely $q^{e}$ choices for $\tau$.
 The number of tuples $(\rho,\tau,\alpha)\in\mathcal{T}'$ of sort 3 hence equals
 \begin{align*}
   & z_{d-j}q^{d-1}q^{e}+\left(x\gs{d-1}{j-2}{q}q^{\binom{j-2}{2}+(j-2)e}-z_{d-j}\right)(q^{d-1}-q^{j-2})q^{e} \\
   & =z_{d-j}q^{j-2}q^{e}+x\gs{d-1}{j-2}{q}q^{\binom{j-2}{2}+(j-2)e}(q^{d-1}-q^{j-2})q^{e}\;.
 \end{align*}
 Consequently,
 \begin{align*}
  |\mathcal{T}| & =q^{\binom{j-1}{2}+(j-1)e}\left(xq^{d+e-1}\gs{d-1}{j-1}{q}+\left(q^{j+e-1}\gs{d-1}{j}{q}-\gs{d-1}{j-1}{q}\right)(x-z_{d-2})\right)                \\
                & \qquad-\left(x\gs{d-1}{j-1}{q}q^{\binom{j-1}{2}+(j-1)e}-z_{d-j-1}\right)q^{j-1}(q^{e}-1)                                                          \\
                & \qquad-z_{d-j}q^{j-2}q^{e}-x\gs{d-1}{j-2}{q}q^{\binom{j-2}{2}+(j-2)e}(q^{d-1}-q^{j-2})q^{e}                                                       \\
                & =xq^{\binom{j-1}{2}+(j-1)e}\left(q^{j+e-1}\left(q^{d-j}\gs{d-1}{j-1}{q}+\gs{d-1}{j}{q}-\gs{d-1}{j-1}{q}\right)+(q^{j-1}-1)\gs{d-1}{j-1}{q}\right. \\
                & \qquad\left.-(q^{d-j+1}-1)\gs{d-1}{j-2}{q}\right)-q^{\binom{j-1}{2}+(j-1)e}\left(q^{j+e-1}\gs{d-1}{j}{q}-\gs{d-1}{j-1}{q}\right)z_{d-2}           \\
                & \qquad+z_{d-j-1}q^{j-1}(q^{e}-1)-z_{d-j}q^{j-2}q^{e}                                                                                              \\
                & \stackrel{IH}{=}xq^{\binom{j}{2}+je}q^{j}\gs{d-1}{j}{q}-q^{\binom{j-1}{2}+(j-1)e}\left(q^{j+e-1}\gs{d-1}{j}{q}-\gs{d-1}{j-1}{q}\right)z_{d-2}     \\
                & \qquad+z_{d-2}\gs{d-2}{j-1}{q}q^{\binom{j-1}{2}+(j-1)e}q^{j-1}(q^{e}-1)-z_{d-2}\gs{d-2}{j-2}{q}q^{\binom{j-2}{2}+(j-2)e}q^{j-2}q^{e}              \\
                & =xq^{\binom{j}{2}+je}q^{j}\gs{d-1}{j}{q}-z_{d-2}q^{\binom{j-1}{2}+(j-1)e}\left(q^{j+e-1}\left(\gs{d-1}{j}{q}-\gs{d-2}{j-1}{q}\right)\right.       \\
                & \qquad\left.+\left(q^{j-1}\gs{d-2}{j-1}{q}+\gs{d-2}{j-2}{q}-\gs{d-1}{j-1}{q}\right)\right)                                                        \\
                & =xq^{\binom{j}{2}+je}q^{j}\gs{d-1}{j}{q}-z_{d-2}q^{\binom{j}{2}+je}q^{j}\gs{d-2}{j}{q}\;.
 \end{align*}
 We find that the number of generators $\sigma\in\mathcal{L}$, with $P\notin\sigma$ and $\dim(\sigma\cap\pi)=d-j-2$, equals
 \[
  \frac{|\mathcal{T}|}{q^{j}}=xq^{\binom{j}{2}+je}\gs{d-1}{j}{q}-z_{d-2}q^{\binom{j}{2}+je}\gs{d-2}{j}{q}\;.
 \]
 Now it follows from the arguments at the beginning of the proof that
 \begin{align*}
  x\gs{d-1}{j}{q}q^{\binom{j}{2}+je} & =z_{d-j-2}+xq^{\binom{j}{2}+je}\gs{d-1}{j}{q}-z_{d-2}q^{\binom{j}{2}+je}\gs{d-2}{j}{q} \\
  \Leftrightarrow \qquad z_{d-j-2}   & =z_{d-2}q^{\binom{j}{2}+je}\gs{d-2}{j}{q}\;.
 \end{align*}
 This finishes the proof of the induction step.
\end{proof}


Before presenting the main classification theorem, we first mention a characterisation result about embedded polar spaces.

\begin{theorem}[{\cite[Theorem 1.7]{dd}}]\label{embeddepolarspace}
 Let $\mathcal{P}$ be a finite classical polar space of rank $d\geq3$ and with parameter $e\geq1$, embedded in a projective space over $\F_{q}$ and let $\mathcal{S}$ be a set of generators of $\mathcal{P}$ such that
 \begin{itemize}
  \item[(i)] for every $i=0,\dots,d$, the number of elements of $\mathcal{S}$ meeting a generator $\pi$ in a $(d-i-1)$-space equals
        \[
         \begin{cases}
          \left(\gs{d-1}{i-1}{q}+q^{i}\gs{d-1}{i}{q}\right)q^{\binom{i-1}{2}+ie-1} & \text{if }\pi\in\mathcal{S}    \\
          (q^{e-1}+1)\gs{d-1}{i-1}{q}q^{\binom{i-1}{2}+(i-1)e}                     & \text{if }\pi\notin\mathcal{S}
         \end{cases}\;;
        \]
  \item[(ii)] for every point $P$ of $\mathcal{P}$ there is a generator $\pi\notin\mathcal{S}$ through $P$;
  \item[(iii)] for every point $P$ of $\mathcal{P}$ and every generator $\pi\notin\mathcal{S}$ through $P$, there are either $(q^{e-1}+1)\gs{d-2}{j}{q}q^{\binom{j}{2}+je}$ generators of $\mathcal{L}$ through $P$ meeting $\tau$ in a $(d-j-2)$-space, for all $j=0,\dots, d-2$, or there are no generators of $\mathcal{L}$ through $P$ meeting $\tau$ in a $(d-j-2)$-space, for all $j=0,\dots, d-2$.
 \end{itemize}
 Then $\mathcal{S}$ is the set of generators of a classical polar space of rank $d$ and with parameter $e-1$ embedded in $\mathcal{P}$.
\end{theorem}

\begin{theorem}\label{mainclassification}
 Let $\mathcal{P}$ be a finite classical polar space of type I of rank $d$ and with parameter $e\geq1$, and let $\mathcal{L}$ be a Cameron-Liebler set of $\mathcal{P}$ with parameter $x$.
 If $x\leq q^{e-1}+1$, then $\mathcal{L}$ is the union of $x$ point-pencils whose vertices are pairwise non-collinear or $x=q^{e-1}+1$ and $\mathcal{L}$ is the set of generators of an embedded polar space of rank $d$ and with parameter $e-1$ (Example~\ref{embedded}).
\end{theorem}
\begin{proof}
 We proceed by induction on $x$. The statement is trivially fulfilled for $x=0$, so we consider the case $x>0$. Let $\pi\in\mathcal{L}$ be a generator. By Lemma~\ref{distance} we know that there are
 \[
  \left((x-1)\gs{d-1}{i-1}{q}+q^{i+e-1}\gs{d-1}{i}{q}\right)q^{\binom{i-1}{2}+(i-1)e}
 \]
 elements of $\mathcal{L}$ meeting $\pi$ in a $(d-i-1)$-space. So, the number of tuples $(Q,\sigma)$ with $\sigma\in\mathcal{L}$ and $Q$ a point in $\pi\cap\sigma$ equals
 \begin{align*}
  N & =\sum^{d-1}_{i=0}\left((x-1)\gs{d-1}{i-1}{q}+q^{i+e-1}\gs{d-1}{i}{q}\right)q^{\binom{i-1}{2}+(i-1)e}\gs{d-i}{1}{q}                                                                   \\
  \intertext{where the term for $i=0$ corresponds to the tuples with $\sigma=\pi$, so}
  N & =\frac{x-1}{q-1}\sum^{d-1}_{i=0}q^{\binom{i-1}{2}+(i-1)e}\gs{d-1}{i-1}{q}\left(q^{d-i}-1\right)+\frac{1}{q-1}\sum^{d-1}_{i=0}q^{\binom{i}{2}+ie}\gs{d-1}{i}{q}\left(q^{d-i}-1\right) \\
    & =\frac{x-1}{q-1}\sum^{d-1}_{i=0}q^{\binom{i}{2}+ie}\gs{d-1}{i}{q}\left(q^{d-i-1}-1\right)+\frac{1}{q-1}\sum^{d-1}_{i=0}q^{\binom{i}{2}+ie}\gs{d-1}{i}{q}\left(q^{d-i}-1\right)       \\
    & =\frac{x-1}{q-1}\left[q^{d-1}\sum^{d-1}_{i=0}q^{\binom{i}{2}}\gs{d-1}{i}{q}q^{i(e-1)}-\sum^{d-1}_{i=0}q^{\binom{i}{2}}\gs{d-1}{i}{q}q^{ie}\right]                                    \\
    & \qquad+\frac{1}{q-1}\left[q^{d}\sum^{d-1}_{i=0}q^{\binom{i}{2}}\gs{d-1}{i}{q}q^{i(e-1)}-\sum^{d-1}_{i=0}q^{\binom{i}{2}}\gs{d-1}{i}{q}q^{ie}\right]                                  \\
    & \stackrel{(\text{Lem. }\ref{qbinomialtheorem})}{=}\frac{x-1+q}{q-1}q^{d-1}\prod^{d-2}_{j=0}\left(q^{j+e-1}+1\right)-\frac{x}{q-1}\prod^{d-2}_{j=0}\left(q^{j+e}+1\right)               \\
    & =\left(\frac{x}{q-1}\left(q^{d-1}(q^{e-1}+1)-(q^{d+e-2}+1)\right)+q^{d-1}(q^{e-1}+1)\right)\prod^{d-3}_{j=0}\left(q^{j+e}+1\right)                                                   \\
    & =\left(x\gs{d-1}{1}{q}+q^{d-1}(q^{e-1}+1)\right)\prod^{d-3}_{j=0}\left(q^{j+e}+1\right)\;.
 \end{align*}
 It is immediate that there is a point $P\in\pi$ which is contained in at least $N\gs{d}{1}{q}^{-1}$ of these tuples, hence in at least $N\gs{d}{1}{q}^{-1}$ elements of $\mathcal{L}$ meeting $\pi$ non-trivially. Assume there is a generator $\tau$ through $P$ that is not contained in $\mathcal{L}$.
 Let $z_{j}$ be the number of elements of $\mathcal{L}$ meeting $\tau$ in a $j$-space through $P$, $0\leq j\leq d-2$. By Lemma~\ref{zj}, $z_{d-j-2}=z_{d-2}\gs{d-2}{j}{q}q^{\binom{j}{2}+je}$, for $j=0,\dots,d-2$. So, the number of elements of $\mathcal{L}$ meeting $\tau$ in a subspace containing $P$ equals
 \begin{align*}
  \sum^{d-2}_{j=0}z_{d-j-2} & =\sum^{d-2}_{j=0}z_{d-2}\gs{d-2}{j}{q}q^{\binom{j}{2}+je}
  \stackrel{(\text{Lem. }\ref{qbinomialtheorem})}{=}z_{d-2}\prod^{d-3}_{j=0}(q^{j+e}+1)\;.
 \end{align*}
 It follows that
 \[
  \frac{\left(x\gs{d-1}{1}{q}+q^{d-1}(q^{e-1}+1)\right)\prod^{d-3}_{j=0}\left(q^{j+e}+1\right)}{\gs{d}{1}{q}}\leq z_{d-2}\prod^{d-3}_{j=0}(q^{j+e}+1)\leq x\prod^{d-3}_{j=0}(q^{j+e}+1)\;.
 \]
 Note that $z_{d-2}\leq x$ by Lemma~\ref{distance}. Now,
 \begin{align*}
                       & \frac{\left(x\gs{d-1}{1}{q}+q^{d-1}(q^{e-1}+1)\right)\prod^{d-3}_{j=0}\left(q^{j+e}+1\right)}{\gs{d}{1}{q}}\leq x\prod^{d-3}_{j=0}(q^{j+e}+1) \\
  \Leftrightarrow\quad & q^{d-1}(q^{e-1}+1)\leq x\left(\gs{d}{1}{q}-\gs{d-1}{1}{q}\right)                                                                              \\
  \Leftrightarrow\quad & q^{e-1}+1\leq x\;.
 \end{align*}
 \par If $x<q^{e-1}+1$, then this is a contradiction. Consequently, in this case all generators through $P$ are contained in $\mathcal{L}$; let $\mathcal{L}'$ be the point-pencil with vertex $P$. Then $\mathcal{L}\setminus\mathcal{L}'$ is a Cameron-Liebler set of $\mathcal{P}$ with parameter $x-1$. This finishes the induction argument.
 Note that no element of $\mathcal{L}\setminus\mathcal{L}'$ contains $P$, hence none of the vertices of the $x-1$ point-pencils in $\mathcal{L}\setminus\mathcal{L}'$ is collinear with $P$.
 \par If $x=q^{e-1}+1$, then either all of the previous inequalities are equalities or some of them are not. In the latter case, we can repeat the argument of the case $x<q^{e-1}+1$ and we find the union of $x$ point-pencils with pairwise non-collinear vertices.
 In the former case we find that $\mathcal{L}$ is a set of $\prod^{d-2}_{j=-1}(q^{j+e}+1)=\prod^{d-1}_{j=0}(q^{j+(e-1)}+1)$ generators admitting the relations of Lemma~\ref{distance} for $x=q^{e-1}+1$ such that for any point there is a generator not in $\mathcal{L}$ through it
 and such that for any point $P$ and generator $\tau\notin\mathcal{L}$ containing $P$ there are either $(q^{e-1}+1)\gs{d-2}{j}{q}q^{\binom{j}{2}+je}$ generators of $\mathcal{L}$ through $P$ meeting $\tau$ in a $(d-j-2)$-space, for all $j=0,\dots, d-2$, or there are no generators of $\mathcal{L}$ through $P$ meeting $\tau$ in a $(d-j-2)$-space, for all $j=0,\dots, d-2$.
 By Theorem~\ref{embeddepolarspace} $\mathcal{L}$ has to be a polar space of rank $d$ and with parameter $e-1$ embedded in $\mathcal{P}$.
\end{proof}

The previous result deals with Cameron-Liebler sets in polar spaces of type I. Now we look at the hyperbolic quadrics of even rank. The classification of the Cameron-Liebler sets in (one class of generators of) $\mathcal{Q}^{+}(3,q)$ is trivial, so we will look at Cameron-Liebler sets in one class of generators of $\mathcal{Q}^{+}(7,q)$.

We recall from Theorem~\ref{distancetransitive} that the automorphism group of a finite classical polar space acts transitively on pairs of disjoint generators. We will now look at triples of pairwise disjoint generators. The following result seems to be known among geometers, but the authors did not find a direct proof. It does follow from~\cite[Proposition 8 and Remark 9]{cp}. However, the proof there uses Segre varieties. For the sake of completeness we will provide here a direct proof.

\begin{lemma}\label{triples}
 The automorphism group of $\mathcal{Q}^{+}(4n-1,q)$, $n\geq1$, acts transitively on triples of pairwise disjoint generators.
\end{lemma}
\begin{proof}
 Without loss of generality we can assume that $\mathcal{Q}=\mathcal{Q}^{+}(4n-1,q)$ is given by the quadratic form $Q(x_{0},\dots,x_{4n-1})=x_{0}x_{2n}+\dots+x_{2n-1}x_{4n-1}$, represented by the matrix $\left(\begin{smallmatrix}0&I\\0&0\end{smallmatrix}\right)$, where all blocks are $(2n\times 2n)$-matrices.
 Let $G$ be the collineation group (automorphism group) of $\mathcal{Q}$; its elements are given by a $(4n\times 4n)$-matrix and an $\F_{q}$-automorphism. Let $H\leq G$ be the subgroup of isometries of $\mathcal{Q}$, the collineations whose corresponding field automorphism is trivial.
 We will prove that for four pairwise disjoint generators $\pi_{1},\pi_{2},\pi_{3},\pi_{4}$ there is an element of $H$ that fixes $\pi_{1}$ and $\pi_{2}$, and that maps $\pi_{3}$ onto $\pi_{4}$.
 Since, by Theorem~\ref{distancetransitive}, $G$ acts transitively on the pairs of pairwise disjoint generators of $\mathcal{Q}$, we can choose $\pi_{1}=\col\left(\left(\begin{smallmatrix}I\\ 0\end{smallmatrix}\right)\right)$ and $\pi_{2}=\col\left(\left(\begin{smallmatrix}0\\ I\end{smallmatrix}\right)\right)$.
 \par The elements of $H$ that fix both $\pi_{1}$ and $\pi_{2}$ are given by the matrices $A=\left(\begin{smallmatrix}A_{1}&0\\0&A_{2}\end{smallmatrix}\right)$, with $A_{1}$ and $A_{2}$ non-singular $(2n\times 2n)$-matrices. For $A$ to be an isometry, moreover we must have
 \[
  \begin{pmatrix}0&I\\0&0\end{pmatrix}=\begin{pmatrix}A_{1}&0\\0&A_{2}\end{pmatrix}\begin{pmatrix}0&I\\0&0\end{pmatrix}\begin{pmatrix}A_{1}&0\\0&A_{2}\end{pmatrix}^{t}=\begin{pmatrix}0&A_{1}A^{t}_{2}\\0&0\end{pmatrix}\;,
 \]
 hence $A_{1}A^{t}_{2}=I$. Any generator that is disjoint to both $\pi_{1}$ and $\pi_{2}$ can be described as $\col\left(\left(\begin{smallmatrix}I\\ X\end{smallmatrix}\right)\right)$ with $X$ a non-singular $(2n\times2n)$-matrix.
 A $(2n-1)$-space $\col\left(\left(\begin{smallmatrix}I\\ X\end{smallmatrix}\right)\right)$ is a generator of $\mathcal{Q}$ if and only if
 \[
  \begin{pmatrix}I&X^{t}\end{pmatrix}\begin{pmatrix}0&I\\I&0\end{pmatrix}\begin{pmatrix}I\\X\end{pmatrix}=0\quad\Leftrightarrow\quad X+X^{t}=0\;,
 \]
 where $X$, $I$ and $0$ are $(2n\times 2n)$-matrices. It follows that $X$ is an alternating matrix. Note that $\left(\begin{smallmatrix}0&I\\I&0\end{smallmatrix}\right)$ is the matrix of the bilinear form associated with $Q$.
 \par Let $C$ and $C'$ be the non-singular alternating $(2n\times2n)$-matrices such that $\pi_{3}=\col\left(\left(\begin{smallmatrix}I\\ C\end{smallmatrix}\right)\right)$ and $\pi_{4}=\col\left(\left(\begin{smallmatrix}I\\ C'\end{smallmatrix}\right)\right)$.
 Since two non-singular alternating matrices of the same dimension are conjugate (up to isomorphism there is only one symplectic polar space of a given rank), there is a non-singular matrix $D$ such that $C'=DCD^{t}$. Now,
 \[
  \begin{pmatrix}
   {(D^{t})}^{-1} & 0 \\0&D
  \end{pmatrix}
  \begin{pmatrix}
   I \\ C
  \end{pmatrix}=
  \begin{pmatrix}
   {(D^{t})}^{-1} \\DC
  \end{pmatrix}
  \text{ and }\col\left(\left(\begin{smallmatrix}{(D^{t})}^{-1}\\
     DC\end{smallmatrix}\right)\right)=\col\left(\left(\begin{smallmatrix}{(D^{t})}^{-1}D^{t}\\
     DCD^{t}\end{smallmatrix}\right)\right)=\col\left(\left(\begin{smallmatrix}I\\
     C'\end{smallmatrix}\right)\right)\;.
 \]
 So, the isometry defined by the matrix $\left(\begin{smallmatrix}{(D^{t})}^{-1}&0\\0&D\end{smallmatrix}\right)$ maps $\pi_{3}$ onto $\pi_{4}$, and it fixes both $\pi_{1}$ and $\pi_{2}$.
 This finishes the proof.
\end{proof}

From~\cite[Proposition 4 and Remark 5]{cp} the analogous statement for the Hermitian varieties $\mathcal{H}(2d-1,q^{2})$ follows. The proof in~\cite{cp} also uses Segre varieties, but it is probably possible to give a proof similar to the above.
We do not go into detail because we do not need this result in this paper. It is immediate that an analogous result cannot hold for the polar spaces $\mathcal{Q}^{+}(4n+1,q)$, $n\geq1$, $\mathcal{Q}(4n,q)$, $n\geq1$, $\mathcal{Q}^{-}(2d+1,q)$, $d\geq1$, and $\mathcal{H}(2d,q^{2})$, $d\geq1$.

In the next lemmata we will consider those hyperbolic quadrics $\mathcal{Q}^{+}(4n-1,q)$ that admit spreads. Recall from Table~\ref{spreadtable} that we know that the hyperbolic quadric $\mathcal{Q}^{+}(4n-1,q)$ admits a spread if $q$ is even.
If $q$ is odd, it admits a spread in case $n=1$ or in case $n=2$ and $q$ is prime or $q\not\equiv1\pmod{3}$; if $q$ is odd and $n\geq3$, or if $q\equiv1\pmod{3}$ is non-prime and odd and  $n=2$, no existence or non-existence results are known, up to the knowledge of the authors. Also, recall that two disjoint generators in $\mathcal{Q}^{+}(4n-1,q)$ belong to the same class.
\par We first prove two lemmata.

\begin{lemma}\label{disjointotwo}
 Let $\mathcal{P}$ be one class of generators of a hyperbolic quadric $\mathcal{Q}^{+}(4n-1,q)$ that admits spreads, and let $\pi$ and $\pi'$ be two disjoint generators of $\mathcal{P}$.
 If $\mathcal{L}$ is a Cameron-Liebler set of $\mathcal{P}$ with parameter $x$, with characteristic vector $\bm{\chi}$, then the number of generators of $\mathcal{L}$ disjoint to both $\pi$ and $\pi'$ equals $(x-{(\bm{\chi})}_{\pi}-{(\bm{\chi})}_{\pi'})q^{n(n-1)}\prod^{n-1}_{i=1}(q^{2i-1}-1)$.
\end{lemma}
\begin{proof}
 The proof is similar to the proof of Lemma~\ref{spreadintersectiontonumberanddisjoint}. Denote the number of spreads of $\mathcal{P}$ containing $i$ given pairwise disjoint generators by $n_{i}$, $i=2,3$.
 By Lemma~\ref{triples} we know these numbers are independent of the chosen generators.
 \par We count the tuples $(\tau,S)$ with $\tau$ a generator of $\mathcal{P}$ disjoint to both $\pi$ and $\pi'$ and $S$ a spread of $\mathcal{P}$ containing $\pi$, $\pi'$ and $\tau$. Using Lemmas~\ref{kmslemma} and~\ref{skewgenerators} we find that
 \[
  n_{2}(q^{2n-1}-1)=n_{3}q^{n(n-1)}\prod^{n}_{i=1}(q^{2i-1}-1)\quad\Leftrightarrow\quad\frac{n_{2}}{n_{3}}=q^{n(n-1)}\prod^{n-1}_{i=1}(q^{2i-1}-1)\;.
 \]
 Now we count the tuples $(\tau,S)$ with $\tau$ a generator of $\mathcal{P}$ disjoint to both $\pi$ and $\pi'$ that is contained in $\mathcal{L}$, and $S$ a spread of $\mathcal{P}$ containing $\pi$, $\pi'$ and $\tau$.
 On the one hand, there are $n_{2}$ spreads containing $\pi$ and $\pi'$, and by Theorem~\ref{characterisationthII}(v) there are $x-{(\bm{\chi})}_{\pi}-{(\bm{\chi})}_{\pi'}$ elements of $\mathcal{L}$ disjoint to both $\pi$ and $\pi'$ contained in this spread.
 On the other hand, for a given $\tau$, there are $n_{3}$ spreads containing $\pi$, $\pi'$ and $\tau$. So, the number of generators of $\mathcal{L}$ disjoint to both $\pi$ and $\pi'$ equals
 \[
  \frac{n_{2}}{n_{3}}(x-{(\bm{\chi})}_{\pi}-{(\bm{\chi})}_{\pi'})=(x-{(\bm{\chi})}_{\pi}-{(\bm{\chi})}_{\pi'})q^{n(n+1)}\prod^{n-1}_{i=1}(q^{2i-1}-1)\;.\qedhere
 \]
\end{proof}

\begin{lemma}\label{nocdisjoint}
 Let $\mathcal{P}$ be one class of generators of a hyperbolic quadric $\mathcal{Q}^{+}(7,q)$ that admits spreads, and let $\mathcal{L}$ be a Cameron-Liebler set of $\mathcal{P}$ with parameter $x$. If
 \[
  (x-1)q^{3}< c(q^{3}+x(q^{2}+q+1))-\binom{c+1}{2}(2q^{2}+x(q+1))\;,
 \] then there are no $c+1$ pairwise disjoint generators in $\mathcal{L}$.
\end{lemma}
\begin{proof}
 We follow the approach from~\cite[Lemma 5.2 and Lemma 6.5]{rsv} and~\cite[Lemma 2.4]{met3}.
 \par Assume that $\pi_{0},\dots,\pi_{c}$ are $c+1$ pairwise disjoint generators in $\mathcal{L}$. We denote the set of generators in $\mathcal{L}$ that have a non-empty intersection with $\pi_{i}$ by $S_{i}$ and the set of generators that have non-empty intersections with both $\pi_{i}$ and $\pi_{j}$ by $S_{ij}$.
 It follows immediately from Theorem~\ref{characterisationthII} and Lemma~\ref{disjointotwo} that $|S_{i}|=q^{3}+x(q^{2}+q+1)$ for all $i=0,\dots,c$, and that $|S_{ij}|=2q^{2}+x(q+1)$ for all $0\leq i\neq j\leq c$.
 \par Using the inclusion-exclusion principle we find that
 \[
  |\mathcal{L}|\geq\left|\bigcup^{c}_{i=0}S_{i}\right|\geq \sum^{c}_{i=0}|S_{i}|-\sum_{0\leq i<j\leq c}|S_{ij}|\;,
 \]
 and hence,
 \begin{align*}
                       & x(q+1)(q^{2}+1)\geq(c+1)(q^{3}+x(q^{2}+q+1))-\binom{c+1}{2}(2q^{2}+x(q+1)) \\
  \Leftrightarrow\quad & (x-1)q^{3}\geq c(q^{3}+x(q^{2}+q+1))-\binom{c+1}{2}(2q^{2}+x(q+1))\;,
 \end{align*}
 a contradiction.
\end{proof}

\par Before presenting the classification theorem for Cameron-Liebler sets of a class of generators of $\mathcal{Q}^{+}(7,q)$, we recall the Erd\H{o}s-Ko-Rado theorem for a class of generators of $\mathcal{Q}^{+}(7,q)$.

\begin{theorem}[{\cite[Theorem 22]{psv} and~\cite[Remark 2.3.2]{deb}}]\label{ekrq+7q}
 A maximal Erd\H{o}s-Ko-Rado set of generators of one class $\mathcal{P}$ of a hyperbolic quadric $\mathcal{Q}^{+}(7,q)$ is a point-pencil (restricted to $\mathcal{P}$) or a base-solid (described in Example~\ref{basesolid}).
\end{theorem}

\begin{theorem}\label{q+7qc}
 Let $\mathcal{P}$ be one class of generators of a hyperbolic quadric $\mathcal{Q}^{+}(7,q)$ that admits spreads, and let $\mathcal{L}$ be a Cameron-Liebler set of $\mathcal{P}$ with parameter $x$.
 If $x\leq\frac{\sqrt{3}-1}{2}q$, then $\mathcal{L}$ is the union of $x$ point-pencils whose vertices are pairwise non-collinear or $\mathcal{L}$ is the union of $x$ base-solids whose centers are pairwise disjoint.
\end{theorem}
\begin{proof}
 We follow the approach from~\cite[Lemma 5.3, Theorem 5.4 and Corollary 5.5]{rsv} and~\cite[Lemma 2.5, Lemma 3.1 and Theorem 3.2]{met3}. Let $m$ be the largest integer such that there exist $m$ pairwise disjoint generators $\pi_{1},\dots,\pi_{m}$ in $\mathcal{L}$. First we prove that $m\leq b=\left(1-\frac{\sqrt{3}}{3}\right)q$, by showing that there cannot be $\left\lfloor b\right\rfloor+1$ pairwise disjoint generators in $\mathcal{L}$.
 We want to use Lemma~\ref{nocdisjoint} and for this we have to prove that $f(\left\lfloor b\right\rfloor)\geq0$ with
 \[
  f(c)=c(q^{3}+x(q^{2}+q+1))-\binom{c+1}{2}(2q^{2}+x(q+1))-(x-1)q^{3}\;.
 \]
 As $f$ is a quadratic function with a negative leading coefficient, it is sufficient to check that $f(b)\geq0$ and $f(b-1)\geq0$. A straightforward calculation shows that
 \begin{align*}
  f(b)\geq0\    & \Leftrightarrow\ x\leq \frac{\sqrt{3}-1}{2}q+f_{1}(q)\quad\text{with}\quad f_{1}(q)=\frac{(2+\sqrt{3})q^{2}+(2\sqrt{3}-3)q}{4q^{2}-(\sqrt{3}-1)q-3+\sqrt{3}}\quad\text{and}                    \\
  f(b-1)\geq0\  & \Leftrightarrow\ x\leq \frac{\sqrt{3}-1}{2}q+f_{2}(q)\quad\text{with}\quad f_{2}(q)=\frac{(5-2\sqrt{3})q^{3}-(6-3\sqrt{3})q^{2}- (3\sqrt{3}-3)q}{4q^{3}+(\sqrt{3}+1)q^{2}+(3\sqrt{3}-3)q+6}\;.
 \end{align*}
 Since $f_{i}(q)\geq0$ for $q\geq2$ and $i=1,2$, we find the statement to be true: we have $m\leq b$.
 \par We denote the set of generators that have a non-empty intersection with both $\pi_{i}$ and $\pi_{j}$ by $S_{ij}$, $1\leq i\neq j\leq m$, and we denote the set of generators in $\mathcal{L}$ that are disjoint to $\pi_{j}$, for all $j\neq i$, by $L_{i}$, $i=1,\dots,m$.
 We know that $|S_{ij}|=2q^{2}+x(q+1)$, for all $1\leq i\neq j\leq m$, by Theorem~\ref{characterisationthII} and Lemma~\ref{disjointotwo}.
 \par It is easy to see that $\pi_{i}\in L_{i}$ for any $i=1,\dots,m$. Moreover any two generators in $L_{i}$ must have a non-empty intersection since there are no $m+1$ pairwise disjoint generators in $\mathcal{L}$ by assumption. Hence, the set $L_{i}$ is an Erd\H{o}s-Ko-Rado set of generators, and consequently, $L_{i}\subseteq\mathcal{L}_{i}$, with $\mathcal{L}_{i}$ a maximal Erd\H{o}s-Ko-Rado set of generators.
 It is straightforward that $L_{i}$ contains all elements of $\mathcal{L}$ meeting $\pi_{i}$ that are not in $\cup_{j\neq i}S_{ij}$. Hence, $|L_{i}|\geq q^{3}+x(q^{2}+q+1)-(m-1)(2q^{2}+x(q+1))$ by Lemma~\ref{distanceII}.
 \par Assume that there exists a generator $\pi\in\mathcal{L}_{i}\setminus L_{i}$ for some $i$. Again by Lemma~\ref{distanceII}, we know that $\pi$ meets $x(q^{2}+q+1)$ generators of $\mathcal{L}$ non-trivially. But, as $\pi$ is contained in the maximal Erd\H{o}s-Ko-Rado set $\mathcal{L}_{i}$, it meets all generators in $L_{i}$ non-trivially. Consequently,
 \begin{align*}
                       & q^{3}+x(q^{2}+q+1)-(m-1)(2q^{2}+x(q+1))\leq x(q^{2}+q+1) \\
  \Leftrightarrow\quad & q^{3}\leq(m-1)(2q^{2}+x(q+1))\;.
 \end{align*}
 However,
 \begin{align*}
  (m-1)(2q^{2}+x(q+1)) & \leq \left(\left(1-\frac{\sqrt{3}}{3}\right)q-1\right)\left(2q^{2}+\frac{\sqrt{3}-1}{2}q(q+1)\right) \\
                       & =q^{3}-\left(\frac{5}{2}-\frac{\sqrt{3}}{3}\right)q^{2}-\frac{\sqrt{3}-1}{2}q                        \\
                       & <q^{3}\;.
 \end{align*}
 We find a contradiction. Hence, $L_{i}=\mathcal{L}_{i}$ for all $i=1,\dots,m$. So, $\mathcal{L}$ is the  union of the maximal Erd\H{o}s-Ko-Rado sets $\mathcal{L}_{i}$, $i=1,\dots,x$, which necessarily have to be pairwise disjoint.
 \par By Theorem~\ref{ekrq+7q} each of these sets $\mathcal{L}_{i}$ is either a point-pencil of $\mathcal{P}$ or else a base-solid. Two point-pencils in $\mathcal{P}$ are disjoint if and only if their vertices are non-collinear. Analogously, two base-solids in $\mathcal{P}$ are disjoint if and only if their centers are disjoint.
 Finally, a point-pencil and a base-solid in $\mathcal{P}$ cannot be disjoint. This finishes the proof of the theorem.
\end{proof}

In the previous theorem we did not discuss the hyperbolic quadric $\mathcal{Q}^{+}(4n-1,q)$ for general $n$, but only the hyperbolic quadric $\mathcal{Q}^{+}(7,q)$. We know that the largest Erd\H{o}s-Ko-Rado sets of $\mathcal{Q}^{+}(4n-1,q)$, $n\geq3$,  are point-pencils.
So, to apply the approach that we used in the previous proof, we need a result that gives an upper bound on the size of the maximal Erd\H{o}s-Ko-Rado sets of $\mathcal{Q}^{+}(4n-1,q)$ different from a point-pencil. However, such a result is not available at the moment.

\subsection*{Summary}

We summarize the classification results of this section in Table \ref{tablesummary}.

\begin{table*}[ht]
 \centering
 \begin{tabular}{|c||c|c|c|}
  \hline
  Polar space                               & condition                                       & classification                   & Reference                        \\ \hline\hline
  $\mathcal{Q}^{-}(2d+1,q)$                 & $x\leq q+1$                                     & $x$ point-pencils or             & Theorem~\ref{mainclassification} \\
                                            &                                                 & Example~\ref{embedded} ($x=q+1$) &                                  \\ \hline
  $\mathcal{Q}(4n,q)$                       & $x\leq 2$                                       & $x$ point-pencils or             & Theorem~\ref{mainclassification} \\
                                            &                                                 & Example~\ref{embedded} ($x=2$)   &                                  \\ \hline
  $\mathcal{Q}(4n+2,q)$, $n\geq2$           & $x=1$                                           & point-pencil or hyperbolic class & Theorem~\ref{par1}               \\  \hline
  $\mathcal{Q}(6,q)$                        & $x=1$                                           & point-pencil, hyperbolic class   & Theorem~\ref{par1}               \\
                                            &                                                 & or Example~\ref{baseplane}       &                                  \\ \hline
  $\mathcal{Q}^{+}(4n+1,q)$                 & $x=1$                                           & point-pencil                     & Theorem~\ref{par1}               \\ \hline
  one class of $\mathcal{Q}^{+}(4n+3,q)$    & $x=1$                                           & point-pencil                     & Theorem~\ref{par1}               \\ \hline
  one class of $\mathcal{Q}^{+}(7,q)$       & $x\leq\frac{\sqrt{3}-1}{2}q$                    & $x$ point-pencils or             & Theorem~\ref{q+7qc}              \\
  that admits a spread                      &                                                 & $x$ base-solids                  &                                  \\ \hline
  $\mathcal{H}(2d-1,q^{2})$                 & $x= 1$                                          & point-pencil                     & Theorem~\ref{par1}               \\ \hline
  $\mathcal{H}(2d,q^{2})$                   & $x\leq q+1$                                     & $x$ point-pencils or             & Theorem~\ref{mainclassification} \\
                                            &                                                 & Example~\ref{embedded} ($x=q+1$) &                                  \\ \hline
  $\mathcal{W}(4n+1,q)$, $q$ even, $n\geq2$ & $x=1$                                           & point-pencil or hyperbolic class & Theorem~\ref{par1}               \\ \hline
  $\mathcal{W}(5,q)$, $q$ even              & $x=1$                                           & point-pencil, hyperbolic class   & Theorem~\ref{par1}               \\
                                            &                                                 & or Example~\ref{baseplane}       &                                  \\ \hline
  $\mathcal{W}(4n+3,q)$, $q$ even           & $x\leq 2$                                       & $x$ point-pencils                & Theorem~\ref{mainclassification} \\
                                            &                                                 & Example~\ref{embedded} ($x=2$)   &                                  \\ \hline
  $\mathcal{W}(4n+3,q)$, $q$ odd            & $x\leq 2$                                       & $x$ point-pencils                & Theorem~\ref{mainclassification} \\ \hline
  $\mathcal{W}(4n+1,q)$, $n\geq2$           & \multicolumn{3}{|c|}{no characterisation known}                                                                       \\ \hline
 \end{tabular}
 \caption{Overview of the classification results}\label{tablesummary}
\end{table*}

\paragraph*{Acknowledgement:} The research of Maarten De Boeck was supported by the BOF--UGent (Special Research Fund of Ghent University). The research of Andrea \v{S}vob has been supported by the Croatian Science Foundation under the Project 6732 and has been supported by a grant from the COST project \emph{Random network coding and designs over GF(q)} (COST IC1104). The authors thank Francesco Pavese (Politecnico di Bari) for his comments on $m$-regular systems. The authors express their gratitude to the referees for their valuable comments.

\end{document}